\documentclass[12pt,twoside,english,a4paper]{article}
\usepackage{babel,amssymb,amsthm}
\usepackage{graphicx}
\usepackage{amsmath}
\usepackage{bm}

\newcommand{\smallfrac}[2]{{\textstyle\frac{#1}{#2}}} 
\newcommand{\jump}[1]{[\![#1]\!]}

\newtheorem{proposition}{Proposition}[section]

\newtheorem{theorem}[proposition]{Theorem}

\numberwithin{equation}{section}

\title{The Costabel-Stephan system of Boundary Integral Equations in the Time Domain}

\date{\today}

\author{Tianyu Qiu \& Francisco--Javier Sayas\footnote{Department of Mathematical Sciences, University of Delaware, Newark DE 19716.  {\tt \{qty,fjsayas\}@udel.edu}. Partially funded by NSF (grant DMS 1216356)}}

\begin{document}

\maketitle

\begin{abstract}
In this paper we formulate a transmission problem for the transient acoustic wave equation as a system of retarded boundary integral equations. We then analyse a fully discrete method using a general Galerkin semidiscretization-in-space and Convolution Quadrature in time. All proofs are developed using recent techniques based on the theory of evolution equations. Some numerical experiments are provided.\\
{\bf AMS Subject classification.} 65N30, 65N38, 65N12, 65N15
\end{abstract}

\section{Introduction}

In this paper we study a boundary integral formulation for the problem of the scattering of a transient acoustic wave by a non-smooth obstacle in two or three space dimensions. The boundary integral formulation is the transient equivalent to Costabel and Stephan's \cite{CoSt:1985} system of integral equations for Helmholtz Transmission Problems. It is a direct formulation (with physical fields as unknowns) involving the four operators of the associated Calder\'on calculus (the entries of the Calder\'on projector) at two different wave speeds. In addition to studying the well-posedness of the problem, we give estimates for Galerkin semidiscretization-in-space, for any choice of the discrete spaces, and analyze the fully discrete method obtained by applying second order Convolution Quadrature to the associated system of semidiscrete delayed integral equations. 

The paper will be written for a single obstacle. The case of multiple separate scatterers (with no common boundaries), each of them possibly with different material properties, is a straightforward extension of this work. The case of layered obstacles might raise new issues from the point of view of analysis, and is the object of future research. The more involved case of obstacles with piecewise constant material properties not organized into layers is due to be considerably more complicated, requiring the use of multitrace spaces following \cite{HiJe:2011, ClHi:2013}. Finally, the choice of modeling the interior of the scatterer using a variational formulation leads to BEM-FEM coupling strategies: \cite{AbJoRoTe:2011, FaMo:2014, BaLuSa:2014}. 

The analysis of the method follows recent techniques \cite{SaySUB, DomSay13, BanLalSaySUB} entirely developed in the time domain, instead of the Laplace domain analysis that stems from the original work of Bamberger and Ha-Duong \cite{BaHa:1986a, BaHa:1986b}. To the best of our knowledge, it is the first time that these techniques are used on a setting that does not involve a purely exterior boundary value problem.We note that part of the Laplace domain analysis for transmission problems using the Costabel-Stephan formulation was included in \cite{LalSay09}, and was extended recently to the case of electromagnetic waves in \cite{ChMo:2014}. As shown in \cite{BanLalSaySUB}, direct time domain analysis provides bounds (for the operators and their discretization) that are sharper than the ones obtained by Laplace domain techniques. The analysis in this paper (like the one in \cite{SaySUB,BanLalSaySUB}) is essentially independent of the spaces used for space discretization. Moreover, the results for semidiscretization are also valid for full Galerkin discretization. For time-stepping we have used a simple BDF2-based Convolution Quadrature method: the analysis for the trapezoidal rule discretization follows very similar arguments and can be easily adapted from what appears in Section 6 of this work and \cite[Section 6]{BanLalSaySUB}. Higher order much less dispersive discretizations are easily attainable using RK-based CQ methods \cite{Banjai:2010, BaMeSc:2012, BaSc:2012, HaSa:2014}. In that case the analysis relies entirely on Laplace domain estimates and the behavior of the bounds with respect to the time variable is still unclear.

A final point deserving clarification is the choice of the Costabel-Stephan format for the formulation of the transmission problem as a system of integral equations. This is a direct formulation using the interior Cauchy data as unknowns and leads to a system that can be considered of equations of the first kind. If we revert to the realm of frequency domain problems, there is a wealth of integral formulations with two goals in mind \cite{BoDoLeTu:2013, GrHeLa:2013, KlMa:1988, KrRo:1978, ToWe:1993, Petersdorff:1989, Zinn:1989}: robustness with respect to frequency (so that there are no resonances in the formulation) and good conditioning of the final system. Both goals can be achieved for smooth scatterers using regularized combined integral formulations, but the analysis of many of these methods is not applicable for non-smooth obstacles. In time domain problems, and with general Lipschitz scatterers, there are additional issues. An important one of them is the difficulty of showing that combined field formulations lead to well-posed problems that stay well-posed after Galerkin discretization. An apparently lone step in this direction appears in \cite{CheMonSB}, by forcing a strong time-regularization in the combined field that worsens the mapping properties of the operator in the time domain.

The paper is structured as follows. In Section \ref{sec:2}, we present the transmission problem and its associated boundary integral formulation. In Section \ref{sec:2B}, we introduce the general Galerkin semidiscretization-in-space and state the two main theorems associated to this part of the discretization: error for the semidiscretization-in-space and a stability theorem that is needed for the study of time discretization. The proofs of these results appear in Sections \ref{sec:4} and \ref{sec:5}. Section \ref{sec:6} deals with full discretization of the problem and Section \ref{sec:7} has some numerical experiments in two dimensions.

\paragraph{Background.} Elementary results on Sobolev spaces (traces, weak normal derivatives, etc) is assumed throughout. Some basic background results in operator-valued distributions are used in the presentation of the equations (Section \ref{sec:2}) and in the quite technical reconciliation of the strong and weak solutions (Section \ref{sec:5}). These aspects can be easily learned from handbooks in advanced tools of applied analysis or, in a simplified version, in \cite{SaySUB}. All needed results from the theory of evolution equations are stated in clear terms before they are used.

\section{Prolegomena and definitions}\label{sec:2}

\paragraph{Geometric setting and problem.}
This paper is concerned with a transmission problem for the wave equation in free space (in dimensions $d=2$ or $d=3$), where a bounded obstacle with homogeneous isotropic material properties is surrounded by a medium with different homogeneous and isotropic properties. In order to set up the problem as soon as possible, let us describe the simplified model we will analyse. Let $\Omega_-\subset \mathbb R^d$ be a bounded Lipschitz domain with boundary $\Gamma$. The exterior domain will be denoted $\Omega_+:=\mathbb R^d\setminus\overline\Omega_-$ and $\boldsymbol\nu$ will be the unit normal vector field on $\Gamma$, pointing from $\Omega_-$ to $\Omega_+$. Two positive parameters $\kappa$ and $c$ will be used to describe the material properties in $\Omega_-$. For sufficiently smooth functions defined in $\mathbb R^d\setminus\Gamma=\Omega_-\cup\Omega_+$, the restrictions to the boundary will be denoted $\gamma^\pm u$. Similarly, $\partial_\nu^\pm u$ will denote the normal derivatives from both sides of $\Gamma$. We will specify their weak definitions later on. We look for the solution to the following transient scattering problem:
\begin{subequations}\label{eq:2.1}
\begin{alignat}{6}
\label{eq:2.1a}
c^{-2} \ddot u(t) = \kappa\Delta u(t) & & \mbox{in $\Omega_-$}, & \quad &\forall t\ge 0,\\
\label{eq:2.1b}
\ddot v(t) =\Delta v(t) & &\mbox{in $\Omega_+$}, & & \forall t\ge 0,\\
\label{eq:2.1c}
\gamma^- u(t)=\gamma^+ v(t)+\beta_0(t) & & \mbox{on $\Gamma$}, & & \forall t\ge 0,\\
\label{eq:2.1d}
\kappa \partial_\nu^- u(t)=\partial_\nu^+ v(t)+\beta_1(t)  & & \mbox{on $\Gamma$}, & & \forall t\ge 0,\\
\label{eq:2.1e}
u(0)=\dot u(0)=0 &\qquad & \mbox{in $\Omega_-$},\\
\label{eq:2.1f}
v(0)=\dot v(0)=0 &\qquad & \mbox{in $\Omega_+$}.
\end{alignat}
\end{subequations}
Let us first clarify some notational aspects of \eqref{eq:2.1}. We need two Sobolev spaces
\[
H^1_\Delta (\Omega_\pm):=\{ w \in H^1(\Omega_\pm)\,:\,\Delta w\in L^2(\Omega_\pm)\},
\]
endowed with their natural norms. We can then define the interior and exterior traces $\gamma^\pm:H^1(\Omega_\pm)\to H^{1/2}(\Gamma)$ and normal derivatives $\partial_\nu^\pm :H^1_\Delta(\Omega_\pm)\to H^{-1/2}(\Gamma)$. We look for sufficiently smooth functions
$u:[0,\infty)\to H^1_\Delta(\Omega_+)$ and $v:[0,\infty)\to H^1_\Delta(\Omega_-)$ satisfying: the differential equations \eqref{eq:2.1a}-\eqref{eq:2.1b} in a classical sense in the time variable and with values in $L^2(\Omega_\pm)$ in the space variables; the transmission condition \eqref{eq:2.1c} in $H^{1/2}(\Gamma)$; the condition \eqref{eq:2.1d} in $H^{-1/2}(\Gamma)$; and finally the initial conditions \eqref{eq:2.1e}-\eqref{eq:2.1f}. The discontinuities across $\Gamma$ are given by functions $\beta_0:[0,\infty)\to H^{1/2}(\Gamma)$ and $\beta_1:[0,\infty)\to H^{-1/2}(\Gamma)$. In principle, these functions are not related to each other, although in a physical setting there is an incident wave $u_{\mathrm{inc}}$ satisfying $\ddot u_{\mathrm{inc}}(t)=\Delta u_{\mathrm{inc}}(t)$ in free space for all $t$, and we take
$\beta_0(t):=\gamma u_{\mathrm{inc}}(t)$ and $\beta_1(t):=\partial_\nu u_{\mathrm{inc}}(t)$.

\paragraph{Strong and weak causality.} Given a function $f:\mathbb R \to X$, where $X$ is any function or operator space, we will say that $f$ is causal when $f\equiv 0$ in $(-\infty,0)$. Given a distribution $f$ in $\mathbb R$ with values in a Banach space $X$, we will say that $f$ is causal when its support is contained in $[0,\infty)$. Note that if $X$ and $Y$ are two Banach spaces,  $A:X\to Y$ is linear and bounded, and $f$ is a causal $X$-valued distribution, then $Af$ defines a causal $Y$-valued distribution. In particular if $X\subset Y$ with bounded embedding, every $X$-valued distribution can be read as a $Y$-valued distribution.
For equality of two distributions$f$ and $g$  with values in a space $X$ we will
use the notation $f=g \quad\mbox{(in $X$)}$, as a reminder of where the equality takes place.
Instead of thinking of the system \eqref{eq:2.1} with strong differentiation and strongly imposed initial conditions, we can write it in the sense of vector-valued distributions: we then look for a causal $H^1_\Delta(\Omega_-)$-valued distribution $u$ and a causal $H^1_\Delta(\Omega_+)$-valued distribution $v$ satisfying
\begin{subequations}\label{eq:2.2}
\begin{alignat}{6}
c^{-2} \ddot u =\kappa\Delta u  & & \qquad & \mbox{(in $L^2(\Omega_-))$},\\
\ddot v =\Delta v & & & \mbox{(in $L^2(\Omega_+)$)},\\
\gamma^- u=\gamma^+ v+\beta_0 & & & \mbox{(in $H^{1/2}(\Gamma)$)},\\
\kappa\partial_\nu^- u =\partial_\nu^+ v+\beta_1 & & & \mbox{(in $H^{-1/2}(\Gamma)$)}.
\end{alignat}
\end{subequations}
Now $\beta_0$ and $\beta_1$ are  causal $H^{1/2}(\Gamma)$- and $H^{-1/2}(\Gamma)$-valued distributions respectively, and differentiation is understood in the sense of distributions in $\mathbb R$. The initial conditions are now implicit in the fact that we assume the distributions to be causal and in the fact that differentiation is taken over the entire real line, and not only on $(0,\infty)$.

\paragraph{Weak form of the Huygens' Potentials.} Given a causal $H^{1/2}(\Gamma)$-valued distribution $\varphi$ and a causal $H^{-1/2}(\Gamma)$-valued distribution $\eta$, and assuming that both of them are Laplace transformable, we consider the following transmission problem: find a causal $H^1_\Delta(\mathbb R^d\setminus\Gamma)$-valued distribution $w$ such that
\begin{subequations}\label{eq:2.3}
\begin{alignat}{6}
m^{-2} \ddot w =\Delta w & &\qquad & \mbox{(in $L^2(\mathbb R^d\setminus\Gamma)$)},\\
\jump{\gamma w}:=\gamma^-w-\gamma^+ w=\varphi & & & \mbox{(in $H^{1/2}(\Gamma)$)},\\
\jump{\partial_\nu w}:=\partial_\nu^-w-\partial_\nu^+ w=\eta& & & \mbox{(in $H^{-1/2}(\Gamma)$)}.
\end{alignat}
\end{subequations}
As shown in \cite[Chapters 2 and 3]{SaySUB}, problem \eqref{eq:2.3} admits a unique solution that can be written using two convolution operators
\begin{equation}\label{eq:2.4}
w=\mathcal S_m *\eta-\mathcal D_m *\varphi.
\end{equation}
Here $\mathcal S_m$ and $\mathcal D_m$ are causal distributions with values in $\mathcal L(H^{-1/2}(\Gamma), H^1_\Delta(\mathbb R^d\setminus\Gamma))$ and $\mathcal L(H^{1/2}(\Gamma), H^1_\Delta(\mathbb R^d\setminus\Gamma))$ respectively, that can be described using their Laplace transforms. Neither the expressions of the Laplace transforms of  $\mathcal S_m$ and $\mathcal D_m$  nor the expression of the convolutional formula \eqref{eq:2.4} for smooth enough (in time) input $\varphi$ and $\eta$ is relevant for what follows. The convolution operators in \eqref{eq:2.3} are called the single and double layer (retarded) potentials. The Laplace transforms will be used later on for the time discretization of the problem using Lubich's Convolution Quadrature techniques. Associated to the layer potentials there are four boundary integral operators, formed by convolution with the following distributions
\begin{subequations}\label{eq:2.5}
\begin{alignat}{6}
\mathcal V_m &:=(\smallfrac12\gamma^-+\smallfrac12\gamma^+)\mathcal S_m =\gamma^\pm \mathcal S_m,  & &\qquad &
	\mbox{valued in $\mathcal L(H^{-1/2}(\Gamma),H^{1/2}(\Gamma))$},\\
\mathcal K_m &:=(\smallfrac12\gamma^-+\smallfrac12\gamma^+)\mathcal D_m, & &\qquad &
	\mbox{valued in $\mathcal L(H^{1/2}(\Gamma),H^{1/2}(\Gamma))$},\\
\mathcal J_m &:=(\smallfrac12\partial_\nu^-+\smallfrac12\partial_\nu^+)\mathcal S_m, & &\qquad &
	\mbox{valued in $\mathcal L(H^{-1/2}(\Gamma),H^{-1/2}(\Gamma))$},\\
\mathcal W_m &:=-(\smallfrac12\partial_\nu^-+\smallfrac12\partial_\nu^+)\mathcal D_m =-\partial_\nu^\pm \mathcal D_m, & &\qquad &
	\mbox{valued in $\mathcal L(H^{-1/2}(\Gamma),H^{-1/2}(\Gamma))$}.
\end{alignat}
\end{subequations}
Note that $\jump{\gamma \cdot}\mathcal S_m=0$ and $\jump{\partial_\nu \cdot}\mathcal D_m=0$, which justifies the three possible ways of defining $\mathcal V_m$ and $\mathcal W_m$.
When $m=1$ in \eqref{eq:2.3} we will omit it in the potential expressions \eqref{eq:2.4} and in the operators \eqref{eq:2.5}

\paragraph{The time domain Costabel-Stephan system.} We return to the weak transmission problem \eqref{eq:2.2}, and choose two unknowns on the boundary
\begin{equation}\label{eq:2.6}
\phi:=\gamma^-v, \qquad \lambda:=\partial_\nu^- v
\end{equation}
to represent the solution of that problem. Using uniqueness arguments tied to the definition of the layer potentials, and thinking of $u$ as extended by zero to $\Omega_+$ and $v$ to be extended by zero to $\Omega_-$, we represent
\begin{equation}\label{eq:2.7}
u=\mathcal S_{m} *\lambda -\mathcal D_{m}*\phi, \qquad v=-\mathcal S*(\kappa\lambda-\beta_1)+\mathcal D*(\phi-\beta_0), \quad \mbox{with $m:=c\sqrt\kappa$.}
\end{equation}
Then, the quantities $\lambda$ and $\phi$ satisfy the system of convolutional boundary integral equations
\begin{equation}\label{eq:2.8}
\left[\begin{array}{cc}
	\mathcal V_{m}+\kappa\mathcal V & -\mathcal K_{m}-\mathcal K \\
	\mathcal J_{m}+\mathcal J   & \mathcal W_{m}+\frac1\kappa\mathcal W
\end{array}\right] * 
\left[\begin{array}{c} \lambda \\ \phi \end{array}\right]
=
\frac12\left[\begin{array}{c} \frac1\kappa\beta_0 \\ \beta_1 \end{array}\right]
+ \left[\begin{array}{cc}
	\mathcal V &-\mathcal K \\
	\frac1\kappa\mathcal J & \frac1\kappa\mathcal W
\end{array}\right] *
\left[\begin{array}{c} \beta_1 \\ \beta_0 \end{array}\right].
\end{equation}
We make these arguments more precise in the next result. Its proof follows from \cite[Section ..]{LalSay09}.

\begin{proposition}
Let $(u,v)$ be a causal solution of \eqref{eq:2.2}. Then the distributions \eqref{eq:2.6} are a causal solution of the system of convolution equations \eqref{eq:2.8} and $(u,v)$ can be represented using the potential expressions \eqref{eq:2.7}. Reciprocally, if $(\lambda,\phi)$ is a causal solution of the system \eqref{eq:2.8} and we define the pair $(u,v)$ using \eqref{eq:2.7}, then $(u,v)$ is a causal solution of \eqref{eq:2.2} and \eqref{eq:2.6} holds.
\end{proposition}

\section{Semidiscretization in space}\label{sec:2B}

\paragraph{Galerkin semidiscretization in space.} We next address the semidiscretization in space of \eqref{eq:2.8} using a Galerkin scheme. Let then $X_h \subset H^{-1/2}(\Gamma)$ and $Y_h\subset H^{1/2}(\Gamma)$ be finite dimensional spaces. We will tag the spaces in the parameter $h$, with no geometric meaning. In order to say that a constant is independent of the choice of the finite dimensional spaces $X_h$ and $Y_h$ we will just say that it is independent of $h$. We will also assume that $\mathbb P_0(\Gamma)\subset X_h$, i.e., the space of constant functions is a subspace of $X_h$. The polar sets of $X_h$ and $Y_h$ are
\begin{alignat*}{6}
X_h^\circ&:=\{ \varphi\in H^{1/2}(\Gamma)\,:\, \langle\mu^h,\varphi\rangle_\Gamma=0\quad\forall \mu^h\in X_h\},\\
Y_h^\circ&:=\{\eta\in H^{-1/2}(\Gamma)\,:\,\langle \eta,\varphi^h\rangle_\Gamma=0\quad\forall \varphi^h\in Y_h\},
\end{alignat*}
where the angled bracket will be used for the $H^{-1/2}(\Gamma)\times H^{1/2}(\Gamma)$ duality product. Following \cite{LalSay09}, we will use polar sets to give shorthand forms of Galerkin-type identities. We will thus write
\[
f-g\in X_h^\circ\quad \mbox{to mean}\quad \langle \mu^h,f\rangle_\Gamma=\langle\mu^h,g\rangle_\Gamma \quad\forall\mu^h\in X_h.
\]
The Galerkin semidiscrete version of \eqref{eq:2.8}, associated to the spaces $X_h$ and $Y_h$ looks for a causal distribution $\lambda^h$ with values in $X_h$ and a causal distribution $\phi^h$ with values in $Y_h$ such that
\begin{equation}\label{eq:2.9}
\left[\begin{array}{cc}
	\mathcal V_{m}+\kappa\mathcal V & -\mathcal K_{m}-\mathcal K \\
	\mathcal J_{m}+\mathcal J   & \mathcal W_{m}+\frac1\kappa\mathcal W
\end{array}\right] * 
\left[\begin{array}{c} \lambda^h \\ \phi^h \end{array}\right]
-
\frac12\left[\begin{array}{c} \frac1\kappa\beta_0 \\  \beta_1\end{array}\right]
- \left[\begin{array}{cc}
	\mathcal V &-\mathcal K \\
	\frac1\kappa\mathcal J & \frac1\kappa\mathcal W
\end{array}\right] *
\left[\begin{array}{c} \beta_1 \\ \beta_0 \end{array}\right]
\in X_h^\circ\times Y_h^\circ.
\end{equation}
A solution of \eqref{eq:2.9} is then used to build approximations of $(u,v)$ using the same potential expressions as in \eqref{eq:2.7}:
\begin{equation}\label{eq:2.10}
u^h:=\mathcal S_m*\lambda^h-\mathcal D_m*\phi^h, \qquad v^h:=-\mathcal S*(\kappa\lambda^h-\beta_1)+\mathcal D*(\phi^h-\beta_0).
\end{equation}

\paragraph{An exotic transmission problem.}
There is a form, based on \cite{LalSay09, Say13, DomSay13}, of writing the Galerkin semidiscretization \eqref{eq:2.9}-\eqref{eq:2.10} by focusing directly on $(u^h,v^h)$. Note that the potentials are defined on both sides of the interface $\Gamma$. They also satisfy the following double transmission problem (with two functions on each side of the boundary and six transmission conditions): $(u^h,v^h)$ are causal $H^1_\Delta(\mathbb R^d\setminus\Gamma)$-valued distributions such that
\begin{subequations}\label{eq:2.11}
\begin{alignat}{6}
c^{-2} \ddot u^h=\kappa \Delta u^h &\qquad && \mbox{(in $L^2(\mathbb R^d\setminus\Gamma)$),}\\
\ddot v^h=\Delta v^h &\qquad && \mbox{(in $L^2(\mathbb R^d\setminus\Gamma)$),}\\
\jump{\gamma u^h}+\jump{\gamma v^h}=\beta_0 & & &\mbox{(in $H^{1/2}(\Gamma)$)},\\
\kappa \jump{\partial_\nu u^h}+\jump{\partial_\nu v^h}=\beta_1 & & &\mbox{(in $H^{-1/2}(\Gamma)$)},\\
\label{eq:2.11e}
\jump{\gamma u^h}\in Y_h & & & \mbox{(in $H^{1/2}(\Gamma)$)},\\
\label{eq:2.11f}
\jump{\partial_\nu u^h} \in X_h& & & \mbox{(in $H^{-1/2}(\Gamma)$)},\\
\label{eq:2.11g}
\gamma^+ u^h-\gamma^- v^h \in X_h^\circ & & & \mbox{(in $H^{1/2}(\Gamma)$)},\\
\label{eq:2.11h}
\kappa\partial_\nu^+ u^h-\partial_\nu^- v^h \in Y_h^\circ & & & \mbox{(in $H^{-1/2}(\Gamma)$)}.
\end{alignat}
\end{subequations}
The fact that there are six transmission conditions (instead of only four) is related to the fact that \eqref{eq:2.11g} and \eqref{eq:2.11h} can be understood as incomplete transmission conditions (only some moments of function are imposed to vanish), and they have to be complemented with dual conditions \eqref{eq:2.11e}-\eqref{eq:2.11f} to make the problem well-posed. As emphasized in \cite{LalSay09}, there is no need for $X_h$ and $Y_h$ to be finite dimensional. Any pair of closed spaces would do. In particular, when we choose $X_h=H^{-1/2}(\Gamma)$ and $Y_h=H^{1/2}(\Gamma)$, the conditions \eqref{eq:2.11e}-\eqref{eq:2.11f} are void (the jumps are in those spaces by definition) and the conditions \eqref{eq:2.11g}-\eqref{eq:2.11h} can be read as $\gamma^+ u^h-\gamma^- v^h=0$ and $\kappa\partial_\nu^+ u^h-\partial_\nu^- v^h=0$ respectively. The following result gives a precise relationship between \eqref{eq:2.11} and \eqref{eq:2.9} and states unique solvability of both. The proof of the result follows from \cite[Section 8]{LalSay09}.

\begin{proposition}\label{prop:2.2}
Let $(u^h,v^h)$ be a causal solution of \eqref{eq:2.11}, and let 
\begin{equation}\label{eq:2.12}
(\lambda^h,\phi^h):=(\jump{\partial_\nu u^h},\jump{\gamma u^h}).
\end{equation} 
Then $(\lambda^h,\phi^h)$ is a causal $(X_h\times Y_h)$-valued distribution solving \eqref{eq:2.9}. Reciprocally, if $(\lambda^h,\phi^h)$ is a causal $(X_h\times Y_h)$-valued distribution satisfying \eqref{eq:2.9} and $(u^h,v^h)$ are defined using \eqref{eq:2.10}, then $(u^h,v^h)$ are causal $H^1(\mathbb R^d\setminus\Gamma)$-valued distributions satisfying \eqref{eq:2.11}. Finally, problem \eqref{eq:2.11} admits a unique causal Laplace transformable solution.
\end{proposition}

\paragraph{The Galerkin solver and the Galerkin error operator.} The Galerkin solver is the linear operator that for given $\beta_0,\beta_1$, solves the problem \eqref{eq:2.9}-\eqref{eq:2.10} or equivalently solves \eqref{eq:2.11}-\eqref{eq:2.12}. The Galerkin error operator is the operator that given a causal distribution $(\lambda,\phi)$ with values in $H^{-1/2}(\Gamma)\times H^{1/2}(\Gamma)$, looks for an $X_h\times Y_h$-valued causal distribution $(\lambda^h,\phi^h)$ satisfying
\begin{subequations}\label{eq:2.13}
\begin{equation}\label{eq:2.13a}
\left[\begin{array}{cc}
	\mathcal V_{m}+\kappa\mathcal V & -\mathcal K_{m}-\mathcal K \\
	\mathcal J_{m}+\mathcal J   & \mathcal W_{m}+\frac1\kappa\mathcal W
\end{array}\right] * 
\left[\begin{array}{c} \lambda^h-\lambda \\ \phi^h-\phi \end{array}\right]\in X_h^\circ\times Y_h^\circ,
\end{equation}
and then outputs
\begin{equation}
(\varepsilon_\lambda^h,\varepsilon^h_\phi):=(\lambda^h-\lambda,\phi^h-\phi),
\end{equation}
as well as the potentials
\begin{equation}\label{eq:2.13c}
e^h_u:=\mathcal S_m * \varepsilon^h_\lambda-\mathcal D_m*\varepsilon^h_\phi,
\qquad
e^h_v:=-\kappa\mathcal S * \varepsilon^h_\lambda+\mathcal D*\varepsilon^h_\phi.
\end{equation}
\end{subequations}
The potentials $e^h_u$ and $e^h_v$ are causal solutions of the following problem:
\begin{subequations}\label{eq:2.14}
\begin{alignat}{6}
c^{-2} \ddot e^h_u=\kappa \Delta e^h_u &\qquad && \mbox{(in $L^2(\mathbb R^d\setminus\Gamma)$),}\\
\ddot e^h_v=\Delta e^h_v &\qquad && \mbox{(in $L^2(\mathbb R^d\setminus\Gamma)$),}\\
\jump{\gamma e^h_u}+\jump{\gamma e^h_v}=0 & & &\mbox{(in $H^{1/2}(\Gamma)$)},\\
\kappa \jump{\partial_\nu e^h_u}+\jump{\partial_\nu e^h_v}=0 & & &\mbox{(in $H^{-1/2}(\Gamma)$)},\\
\jump{\gamma e^h_u}+\phi\in Y_h & & & \mbox{(in $H^{1/2}(\Gamma)$)},\\
\jump{\partial_\nu e^h_u}+\lambda \in X_h& & & \mbox{(in $H^{-1/2}(\Gamma)$)},\\
\gamma^+ e^h_u-\gamma^- e^h_v \in X_h^\circ & & & \mbox{(in $H^{1/2}(\Gamma)$)},\\
\kappa\partial_\nu^+ e^h_u-\partial_\nu^- e^h_v \in Y_h^\circ & & & \mbox{(in $H^{-1/2}(\Gamma)$)}.
\end{alignat}
\end{subequations}
(The problem is very similar to \eqref{eq:2.11}. The only modification is which of the transmission conditions are non-homogeneous.)
Using Laplace domain techniques as in \cite[Chapter 5]{SaySUB}, the Galerkin solver and the Galerkin error operator can be described as convolution operators. We next state the two main theorems of this part of the paper. Their proofs will be the goal of the next two sections. Some notation is first needed. For a given Hilbert space $X$, we consider the spaces
\begin{alignat*}{6}
\mathcal C^k_+(X) &:= \{ f\in \mathcal C^{k}(\mathbb R;X)\,:\, \mathrm{supp}\,f\subset [0,\infty)\},\\
\mathcal W^k_+(X)&:= \{ f\in \mathcal C^{k-1}(\mathbb R;X)\,:\, \mathrm{supp}\, f\subset [0,\infty), \, f^{(k)}\in L^1_{\mathrm{loc}}(\mathbb R;X)\}.
\end{alignat*}
We also consider the cummulative seminorms
\[
H_k(f,t;X):=\sum_{j=0}^k \int_0^t \| f(\tau)\|_X\mathrm d\tau
\]
and the antidifferentiation operator
\[
(\partial^{-1}f)(t):=\int_0^t f(\tau)\mathrm d\tau.
\]

\begin{theorem}\label{th:A}
Let $\beta_0\in \mathcal W_+^2(H^{1/2}(\Gamma))$ and $\beta_1\in \mathcal W^1_+(H^{-1/2}(\Gamma))$, and let $(u^h,v^h,\lambda^h,\phi^h)$ be the solution of \eqref{eq:2.9}-\eqref{eq:2.10}. Then
\[
u^h,v^h \in \mathcal C^1_+(L^2(\mathbb R^d))\cap \mathcal C^0_+(H^1(\mathbb R^d\setminus\Gamma)), \quad \
\phi^h \in \mathcal C^0_+(H^{1/2}(\Gamma))
\]
and
\begin{alignat*}{6}
& \| u^h(t)\|_{1,\mathbb R^d\setminus\Gamma}
+
\| v^h(t)\|_{1,\mathbb R^d\setminus\Gamma}
+
\|\phi^h(t)\|_{1/2,\Gamma} \qquad & \\
 & \qquad \le 
C \left(
H_3(\partial^{-1}\beta_0, t; H^{1/2}(\Gamma))
+
H_2(\partial^{-1}\beta_1, t; H^{-1/2}(\Gamma))\right).
\end{alignat*}
If $\beta_0\in \mathcal W_+^3(H^{1/2}(\Gamma))$ and $\beta_1\in \mathcal W^2_+(H^{-1/2}(\Gamma))$, then
$
\lambda^h \in \mathcal C^0_+(H^{-1/2}(\Gamma))
$
and
\[
\|\lambda^h (t)\|_{-1/2,\Gamma} \le 
C \left(
H_3(\beta_0, t; H^{1/2}(\Gamma))
+
H_2(\beta_1, t; H^{-1/2}(\Gamma))\right).
\]
\end{theorem}

\begin{theorem}\label{th:B}
Let $\phi\in \mathcal W_+^2(H^{1/2}(\Gamma))$ and $\lambda\in \mathcal W^1_+(H^{-1/2}(\Gamma))$, and let $(e^h_u,e^h_v,\varepsilon_\lambda^h,\varepsilon_\phi^h)$ be the solution of \eqref{eq:2.13}. Then
\[
e^h_u,e^h_v \in \mathcal C^1_+(L^2(\mathbb R^d))\cap \mathcal C^0_+(H^1(\mathbb R^d\setminus\Gamma)), \quad \
\varepsilon_\phi^h \in \mathcal C^0_+(H^{1/2}(\Gamma))
\]
and
\begin{alignat*}{6}
&
\| e_u^h(t)\|_{1,\mathbb R^d\setminus\Gamma}
+
\| e_v^h(t)\|_{1,\mathbb R^d\setminus\Gamma}
+
\|\varepsilon_\phi^h(t)\|_{1/2,\Gamma} \\
& \qquad  \le 
C \left(
H_3(\partial^{-1}\phi, t; H^{1/2}(\Gamma))
+
H_2(\partial^{-1}\lambda, t; H^{-1/2}(\Gamma))
\right).
\end{alignat*}
If $\phi\in \mathcal W_+^3(H^{1/2}(\Gamma))$ and $\lambda\in \mathcal W^2_+(H^{-1/2}(\Gamma))$, then
$
\varepsilon_\lambda^h \in \mathcal C^0_+(H^{-1/2}(\Gamma))
$
and
\[
\|\varepsilon_\lambda^h (t)\|_{-1/2,\Gamma} \le 
C \left(
H_3(\phi, t; H^{1/2}(\Gamma))
+
H_2(\lambda, t; H^{-1/2}(\Gamma))\right).
\]
\end{theorem}

\paragraph{Error estimate for Galerkin semidiscretization.} Note that Theorem \ref{th:B} is de facto an error estimate for Galerkin semidiscretization in space. The reason is simple: if $\phi$ and $\lambda$ take values in $Y_h$ and $X_h$ respectively, then all error quantities ($e^h_u$, $e^h_v$, $\varepsilon^h_\phi$, and $\varepsilon^h_\lambda$) vanish. Therefore, if $\Pi_{Y_h}:H^{1/2}(\Gamma)\to Y_h$ and $\Pi_{X_h}:H^{-1/2}(\Gamma)\to X_h$ are orthogonal projectors, we can place $\phi-\Pi_{Y_h}\phi$ and $\lambda-\Pi_{X_h}\lambda$ instead of $\phi$ and $\lambda$
in the right-hand-sides of the bounds in Theorem \ref{th:B}.

\section{Short term analysis in cut-off domains}\label{sec:4}

\paragraph{The cut-off problems.} Let $B_0:=B(\mathbf 0;R)=\{ \mathbf x\in \mathbb R^d\,:\, |\mathbf x|< R\}$ be such that $\overline\Omega_-\subset B_0$. Let then
\[
B_T:=B(\mathbf 0;R+\max\{1,m\} T). 
\]
We will now need two more trace operators
\[
\gamma_T:H^1(B_T\setminus\Gamma) \to H^{1/2}(\partial B_T), \qquad \partial_T^\nu: H^1_\Delta (B_T\setminus\Gamma) \to H^{-1/2}(\partial B_T),
\]
and the space
\[
H^1_{\partial B_T}(B_T\setminus\Gamma):=\{ w\in H^1(B_T\setminus\Gamma)\,:\, \gamma_T w=0\}.
\]
We next present cut-off strong versions of problems \eqref{eq:2.11} and \eqref{eq:2.14}. To unify notation of both problems and to make comparisons with \cite{LalSay09} simpler, the pairs $(u^h,v^h)$ in \eqref{eq:2.11} and $(e^h_u,e^h_v)$ are renamed $(u,u^\star)$. From this moment on starred variables will always correspond to the original exterior problem.
The data for the cut-off Galerkin solver are smooth enough (this will be precised later on) functions $\beta_0:[0,\infty)\to H^{1/2}(\Gamma)$ and $\beta_1:[0,\infty)\to H^{-1/2}(\Gamma)$. We look for two functions $u,u^\star:[0,\infty) \to H^1_\Delta(B_T\setminus\Gamma)$ satisfying, for all $t\ge 0$,
\begin{subequations}\label{eq:3.1}
\begin{alignat}{6}
& c^{-2} \ddot u(t)=\kappa \Delta u(t) \quad \mbox{in $ B_T\setminus\Gamma$}, &\qquad & 
\ddot u^\star(t)=\Delta u^\star(t) \quad \mbox{in $B_T\setminus\Gamma$}, \\
& \gamma_T u(t)=0, &\qquad &
\gamma_T u^\star (t) = 0,\\
&\gamma^+ u(t)-\gamma^- u^\star(t)\in X_h^\circ
&\qquad &
\kappa\partial_\nu^+ u(t)-\partial_\nu^- u^\star(t)\in Y_h^\circ,
\end{alignat}
vanishing initial conditions
\begin{equation}
u(0)=0,\quad \dot u(0)=0, \quad u^\star(0)=0, \quad \dot u^\star(0)=0,
\end{equation}
\end{subequations}
and two more transmission conditions
\begin{subequations}\label{eq:3.2}
\begin{alignat}{6}
& \jump{\gamma u}(t)+\jump{\gamma u^\star}(t)=\beta_0(t),
&\qquad&
\kappa\jump{\partial_\nu u}(t) +\jump{\partial_\nu u^\star}(t)=\beta_1(t),\\
&
\jump{\gamma u}(t)\in Y_h,
&\qquad&
\jump{\partial_\nu u}(t)\in X_h.
\end{alignat}
\end{subequations}
The cut-off Galerkin error operator equations, corresponding to \eqref{eq:2.14}, susbtitute \eqref{eq:3.2} by 
\begin{subequations}\label{eq:3.3}
\begin{alignat}{6}
& \jump{\gamma u}(t)+\jump{\gamma u^\star}(t)=0,
&\qquad&
\kappa\jump{\partial_\nu u}(t) +\jump{\partial_\nu u^\star}(t)=0,\\
&
\jump{\gamma u}(t)+\phi(t)\in Y_h,
&\qquad&
\jump{\partial_\nu u}(t)+\lambda(t)\in X_h,
\end{alignat}
\end{subequations}
where $\phi:[0,\infty)\to H^{1/2}(\Gamma)$ and $\lambda:[0,\infty)\to H^{-1/2}(\Gamma)$.

The plan of this section and the next is very similar to what appears in \cite[Chapters 7 \& 8]{SaySUB} and in \cite{BanLalSaySUB}. We will just sketch the different results, pointing out novelties and new difficulties that did not arise in simpler situations. The goal is simple. We first study problems \eqref{eq:3.1}, completed with \eqref{eq:3.2} or \eqref{eq:3.3}, and the quantities $\jump{\gamma u}$ and $\jump{\partial_\nu u}$ as $t$ grows, showing clear dependence on the cut-off section $T$ in all bounds. We next identify the solution of these two problems with the solution of \eqref{eq:2.11} or \eqref{eq:2.14} (and therefore of the integral system \eqref{eq:2.9} or \eqref{eq:2.13a}) in the time interval $[0,T]$. Finally, we use this identification to provide bounds for the solution of \eqref{eq:2.9}-\eqref{eq:2.10} or \eqref{eq:2.13} as a function of $t$.

\paragraph{Spaces, norms, and integration by parts formulas.} We need to introduce some spaces and operators in order to describe the solution of \eqref{eq:3.1}. Let
\begin{subequations}
\begin{equation}
H:=L^2(B_T)^2
\end{equation}
be endowed with the inner product
\begin{equation}
((u,u^*),(v,v^*))_H:= c^{-2} (u,v)_{B_T}+(u^\star,v^\star)_{B_T}.
\end{equation}
Consider also its subspace
\begin{equation}
V:=\{ (v,v^\star)\in H^1_{\partial B_T}(B_T\setminus\Gamma)^2\,:\,  \jump{\gamma v}=-\jump{\gamma v^\star}\in Y_h,\quad\gamma^+ v-\gamma^- v^\star\in X_h^\circ\},
\end{equation}
endowed with the inner product
\begin{equation}\label{eq:3.2d}
[(u,u^\star),(v,v^\star)]:=\kappa (\nabla u,\nabla v)_{B_T\setminus\Gamma}+(\nabla u^\star,\nabla v^\star)_{B_T\setminus\Gamma}.
\end{equation}
The proof that \eqref{eq:3.2d} is an inner product in $V$ is an easy exercise. Consider next the space
\begin{equation}
D(A):=\left\{ (v,v^\star)\in V\,:\, 
\begin{array}{l}
\Delta v, \Delta v^\star\in L^2(B_T\setminus\Gamma),\\
\jump{\partial_\nu v}=-\kappa^{-1}\jump{\partial_\nu v^\star}\in X_h, \,\kappa \partial_\nu^+ v-\partial_\nu^- v^\star\in Y_h^\circ
\end{array}
\right\},
\end{equation}
and the operator $A:D(A)\to H$ given by $A(v,v^\star):=(c^2\kappa\Delta v,\Delta v^\star)$.
\end{subequations}

\begin{proposition}\label{prop:3.1}
The following properties hold:
\begin{itemize}
\item[{\rm (a)}] The inclusion of $V$ into $H$ is dense and compact.
\item[{\rm (b)}] There exists $C_T$ independent of $h$ such that
\begin{equation}\label{eq:3.33}
\| (v,v^\star)\|_H \le C_T \| (v,v^\star)\|_V\qquad \forall (v,v^\star)\in V.
\end{equation}
\item[{\rm (c)}] The following Green identity holds:
\begin{equation}
(A(u,u^\star),(v,v^\star))_H+[(u,u^\star),(v,v^\star)]=0\qquad \forall (u,u^\star)\in D(A), \quad (v,v^\star)\in V.
\end{equation}
\item[{\rm (d)}] The operator $I-A:D(A)\to H$ is surjective.
\end{itemize}
\end{proposition}

\begin{proof}
The proof of  (a) is straightfoward. To verify (b), note that there exists $C_T>0$ such that \eqref{eq:3.33} holds for all $(v,v^\star)$ in the space
\[
\{ (v,v^\star)\in H^1_{\partial B_T}(B_T\setminus\Gamma)^2\,:\, \jump{\gamma v}=-\jump{\gamma v^\star}, \,\langle 1, \gamma^+v-\gamma^- v^\star\rangle=0\}.
\]
This can be proved with a simple compactness argument. To prove (c) it is enough to note that some simple algebraic manipulations yield:
\begin{alignat}{6}
\nonumber
& \hspace{-2cm}\kappa\left(\langle \partial_\nu^- u,\gamma^- v\rangle-\langle\partial_\nu^+ u,\gamma^+ v\rangle\right)+\langle\partial_\nu^- u^\star,\gamma^-v^\star\rangle-\langle \partial_\nu^+ u^\star,\gamma^+ v^\star\rangle \\
\label{eq:3.3B}
= \, & \kappa\langle \jump{\partial_\nu u},\gamma^+ v-\gamma^- v^\star\rangle
	+\langle \kappa\partial_\nu^+u-\partial_\nu^- u^\star,\jump{\gamma v}\rangle\\
	&  +\langle\partial_\nu^- u^\star,\jump{\gamma v}+\jump{\gamma v^\star}\rangle 
	+\langle \kappa \jump{\partial_\nu u}+\jump{\partial_\nu u^\star},\gamma^+ v^\star\rangle.
\nonumber
\end{alignat}
Let us finally sketch the proof of (d). Given $(f,f^\star)\in H$, we solve the coercive variational problem
\begin{subequations}\label{eq:3.3C}
\begin{alignat}{6}
&  (u,u^\star)\in V,\\
&  ((u,u^\star),(v,v^\star))_H+[(u,u^\star),(v,v^\star)]=((f,f^\star),(v,v^\star))_H \quad\forall (v,v^\star)\in V.
\end{alignat}
\end{subequations}
Testing equations \eqref{eq:3.3C} with $(v,v^\star)\in \mathcal D(B_T\setminus\Gamma)^2$ (the space of smooth compactly supported functions), it follows that
\begin{equation}\label{eq:3.3D}
-c^2\kappa \Delta u+u=f\quad \mbox{in $B_T\setminus\Gamma$}, \qquad -\Delta u^\star+u^\star=f^\star \quad\mbox{in $B_T\setminus\Gamma$}.
\end{equation}
Plugging \eqref{eq:3.3D} in \eqref{eq:3.3C}, using the definition of weak normal derivatives and \eqref{eq:3.3B}, we then show that
\begin{equation}\label{eq:3.3E}
\kappa\langle \jump{\partial_\nu u},\gamma^+ v-\gamma^- v^\star\rangle
	+\langle \kappa\partial_\nu^+u-\partial_\nu^- u^\star,\jump{\gamma v}\rangle
	+\langle \kappa \jump{\partial_\nu u}+\jump{\partial_\nu u^\star},\gamma^+ v^\star\rangle=0\quad\forall (v,v^\star)\in V.
\end{equation}
Finally, the map from $H^1_{\partial B_T}(B_T\setminus\Gamma)^2\to H^{1/2}(\Gamma)^4$ given by
\[
(v,v^\star) \longmapsto 
\left[\begin{array}{c}
\gamma^+v-\gamma^-v^\star\\
\jump{\gamma v}\\
\gamma^+v^\star\\
\jump{\gamma v}+\jump{\gamma v^\star}
\end{array}\right]
=
\left[\begin{array}{cccc}
	0 & 1 & -1 & 0 \\
	1 & -1 & 0 & 0 \\
	0 & 0 & 0 & 1 \\
	1 & -1 & 1 & -1
\end{array}\right]
\left[\begin{array}{c}
	\gamma^ - v \\
	\gamma^+ v \\
	\gamma^- v^\star \\
	\gamma^+ v^\star
\end{array}\right]
\]
is surjective, and therefore that the map
\[
V\ni (v,v^\star)\longmapsto 
(\gamma^+v-\gamma^-v^\star,\jump{\gamma v},\gamma^+v^\star)\in
 X_h^\circ\times Y_h \times H^{1/2}(\Gamma)
\]
is also surjective. This implies that the condition \eqref{eq:3.3E} yields what is necessary to ensure that $(u,u^\star)\in D(A)$ and therefore $(u,u^\star)-A(u,u^\star)=(f,f^\star)$ by \eqref{eq:3.3D}.
\end{proof}

\begin{proposition}[Second order differential equations]\label{prop:3.2}
Let $D(A)\subset V \subset H$  be Hilbert spaces and $A:D(A)\to H$ be a bounded linear operator. Assume that {\rm (a)-(d)} in Proposition \ref{prop:3.1} are satisfied. If $f\in \mathcal C([0,\infty);V)$, then the initial value problem
\begin{equation}\label{eq:3.5}
\ddot u(t)=Au(t)+f(t),\quad t\ge 0, \qquad u(0)=\dot u(0)=0.
\end{equation}
has a unique solution
\begin{equation}\label{eq:3.55}
u\in \mathcal C([0,\infty);D(A))\cap \mathcal C^1([0,\infty);V)\cap \mathcal C^2([0,\infty);H).
\end{equation}
For all $t\ge 0$, this solution satisfies
\begin{subequations}\label{eq:3.56}
\begin{alignat}{6}
\label{eq:3.56a}
C_T^{-1} \| u(t)\|_H \le \| u(t)\|_V  & \le  \int_0^t \| f(\tau)\|_H\mathrm d\tau ,\\
\|\dot u(t)\|_V  &\le  \int_0^t \| f(\tau)\|_V \mathrm d\tau,\\
\label{eq:3.56c}
\|\dot u(t)\|_H & \le   \int_0^t \| f(\tau)\|_H \mathrm d \tau,\\
\| Au(t)\|_H  & \le  \int_0^t \| f(\tau)\|_V \mathrm d\tau.
\end{alignat}
If $f\in \mathcal C^1([0,\infty);H)$ and $f(0)=0$, then problem \eqref{eq:3.5} has a unique solution in \eqref{eq:3.55} satisfying, for all $t$
\begin{alignat}{6}
\|\dot u(t)\|_V  & \le \int_0^t \| \dot f(\tau)\|_H \mathrm d\tau,\\
\| Au(t)\|_H &\le 2 \int_0^t \|\dot f(\tau)\|_H \mathrm d \tau,
\end{alignat}
as well as \eqref{eq:3.56a} and \eqref{eq:3.56c}.
\end{subequations}
\end{proposition}

\begin{proof}
The proofs for these results can be found in \cite[Chapter 6]{SaySUB} (see also \cite[Appendix A]{DomSay13} for an earlier version). Note that those results might be attainable with more abstract tools of functional calculus and semigroup theory and that no claim is made to their originality.
\end{proof}

\paragraph{The lifting operator.} The non-homogeneous transmission conditions in \eqref{eq:3.2} and \eqref{eq:3.3} have to be lifted in order to make the problems \eqref{eq:3.1}-\eqref{eq:3.2}-\eqref{eq:3.3} analyzable with the tools of equation \eqref{eq:3.5}. 
Liftings of non-homogeneous boundary conditions are done using steady-state problems. We can gather them all in a single transmission problem:
\begin{subequations}\label{eq:3.22}
\begin{alignat}{6}
& c^{-2}  w=\kappa \Delta w \quad \mbox{in $ B_T\setminus\Gamma$}, &\qquad & 
w^\star=\Delta w^\star \quad \mbox{in $B_T\setminus\Gamma$}, \\
& \gamma_T w=0, &\qquad &
\gamma_T w^\star= 0,\\
&\gamma^+ w-\gamma^- w^\star \in X_h^\circ
&\qquad &
\kappa\partial_\nu^+ w-\partial_\nu^- w^\star\in Y_h^\circ,\\
& \jump{\gamma w}+\jump{\gamma w^\star}=\beta_0,
&\qquad&
\kappa\jump{\partial_\nu w} +\jump{\partial_\nu w^\star}=\beta_1,\\
&
\jump{\gamma w}+\phi\in Y_h,
&\qquad&
\jump{\partial_\nu w}+\lambda\in X_h.
\end{alignat}
\end{subequations}
The solution of \eqref{eq:3.22} will be written as $(w,w^\star):=L(\beta_0,\phi,\beta_1,\lambda)$. The boundedness of this operator is stated in the next result. Note that $L(0,0,\beta_1,\lambda)$ takes values in  $V$.

\begin{proposition}\label{prop:3.3}
There exists $C$ independent of $h$ and $T$ such that the solution of \eqref{eq:3.22} can be bounded by
\[
\| (w,w^\star)\|_{1,B_T\setminus\Gamma}+\| (\Delta w,\Delta w^\star)\|_{B_T\setminus\Gamma}\le C
\Big( \|\beta_0\|_{1/2,\Gamma}+\|\phi\|_{1/2,\Gamma}+\| \beta_1\|_{-1/2,\Gamma}+\|\lambda\|_{-1/2,\Gamma}\Big).
\]
\end{proposition}

\begin{proof}
This follows from an easy (but careful) variational argument. Consider the bilinear form
\[
a((u,u^\star),(v,v^\star)):= [(u,u^\star),(v,v^\star)]+((u,u^\star),(v,v^\star))_H.
\]
The solution of \eqref{eq:3.22} is the same as the solution of
\begin{subequations}\label{eq:3.23}
\begin{alignat}{6}
	 & (w,w^\star)\in H^1_{\partial B_T}(B_T\setminus\Gamma)^2,\\
	 & \gamma^+ w-\gamma^- w^\star\in X_h^\circ, \quad \jump{\gamma w}+\phi\in Y_h, 
	\quad \jump{\gamma w}+\jump{\gamma w^\star}=\beta_0,\\
	 & a((w,w^\star),(v,v^\star))=
		\langle \beta_1,\gamma^+ v^\star\rangle_\Gamma-\langle\lambda,\gamma^+v-\gamma^-v^\star\rangle_\Gamma  
		\qquad \forall (v,v^\star)\in V.
\end{alignat}
\end{subequations}
By substracting $(w_0,w^\star_0)$ satisfying
\[
\gamma^+w_0=\gamma^- w_0^\star=0, \qquad \gamma^- w_0=-\phi, \qquad \gamma^+ w_0^\star=-\beta_0-\phi, \qquad w_0=w_0^\star\equiv 0 \mbox{ in $B_T\setminus B_0$},
\]
problem \eqref{eq:3.23} can be transformed into a coercive variational problem in $V$. The $H^1(B_T\setminus\Gamma)$ bound on $(w,w^\star)$ follows from this variational formulation. Finally the bound for the Laplacians follow from the first equation of \eqref{eq:3.22}.
\end{proof}

\paragraph{Additional notation.} In the space $H^1(B_T\setminus\Gamma)$, we consider the usual Sobolev norm 
\[
\| u\|_{1,B_T\setminus\Gamma}^2:=\| u\|_{B_T\setminus\Gamma}^2+\| \nabla u\|_{B_T\setminus\Gamma}^2,
\]
while in $H^1_\Delta(B_T\setminus\Gamma)$, we can consider the seminorm
\[
| u|_{\Delta,B_T\setminus\Gamma}^2:=\| \nabla u\|_{B_T\setminus\Gamma}^2+\| \Delta u\|_{B_T\setminus\Gamma}^2,
\]
which is equal to the $\mathbf H(\mathrm{div},B_T\setminus\Gamma)$ norm of $\nabla u$. For a Hilbert space $X$, we will also consider the spaces
\[
\mathcal C^k_0([0,\infty);X):=\{ f\in \mathcal C^k([0,\infty);X)\,:\, f^{(j)}(0)=0, \quad j\le k-1\}.
\]

\begin{proposition}\label{prop:3.4}
Let $\beta_0\in \mathcal C^3_0([0,\infty);H^{1/2}(\Gamma))$ and $\beta_1\in \mathcal C^2_0([0,\infty);H^{-1/2}(\Gamma))$. Then the solution of \eqref{eq:3.1} and \eqref{eq:3.2} is in the space
\[
\mathcal C^2([0,\infty);L^2(B_T)^2)\cap \mathcal C^1([0,\infty);H^1(B_T\setminus\Gamma)^2)\cap \mathcal C([0,\infty);H^1_\Delta(B_T\setminus\Gamma)^2)
\]
and satisfies for all $t\ge 0$
\begin{eqnarray*}
|u(t)|_{\Delta,B_T\setminus\Gamma}+ |u^\star(t)|_{\Delta,B_T\setminus\Gamma} 
	& \le & C\big( H_3(\beta_0,t; H^{1/2}(\Gamma))+ H_2(\beta_1,t;H^{-1/2}(\Gamma))\big),\\
\| \dot u(t)\|_{1,B_T\setminus\Gamma}+\|\dot u^\star(t)\|_{1,B_T\setminus\Gamma} 
	& \le & C\big( H_3(\beta_0,t; H^{1/2}(\Gamma))+ H_2(\beta_1,t;H^{-1/2}(\Gamma))\big).
\end{eqnarray*}
\end{proposition}

\begin{proof}
The key is the decomposition of the solution in the form
\[
\underline u=(u,u^\star) =\underline u_0+\underline u_1+\underline w_0+\underline w_1,
\]
where 
\[
\underline u_0:=L(\beta_0,0,0,0)\in \mathcal C^3_0([0,\infty);H), \qquad \underline u_1:=L(0,0,\beta_1,0)\in \mathcal C^2_0([0,\infty);V)
\]
(see Proposition \ref{prop:3.3}), and
\begin{subequations}
\begin{alignat}{6}
\label{eq:4.1a}
\underline{\ddot w}_0= A\underline w_0+\underline f_0, & \qquad & \underline w_0(0)=\underline{\dot w}_0(0)=0, 
	&\qquad  &  \underline f_0:=\underline u_0-\underline{\ddot u}_0=L(\beta_0-\ddot\beta_0,0,0,0),\\
\label{eq:4.1b}
\underline{\ddot w_1}= A\underline w_1+\underline f_1, & \qquad & \underline w_1(0)=\underline{\dot w_1}(0)=0, 
	&\qquad  &  \underline f_1:=\underline u_1-\underline{\ddot u}_1=L(0,0,\beta_1-\ddot\beta_1,0).
\end{alignat} 
\end{subequations}
It is clear that $\underline f_0\in \mathcal C^1_0([0,\infty);H)$ and $\underline f_1\in \mathcal C([0,\infty);V)$. The solution of \eqref{eq:4.1a} and \eqref{eq:4.1b} fits as a particular instance of Proposition \ref{prop:3.2}. The bounds follow from Proposition \ref{prop:3.3} and from the inequalities \eqref{eq:3.56} in Proposition \ref{prop:3.2}. (Note that Proposition \ref{prop:3.3} is used to bound different norms of $\underline f_j$ in terms of norms of  $\beta_j$.)
\end{proof}

\begin{proposition}
Let $\phi\in \mathcal C^3_0([0,\infty);H^{1/2}(\Gamma))$ and $\lambda\in \mathcal C^2_0([0,\infty);H^{-1/2}(\Gamma))$. Then the solution of \eqref{eq:3.1} and \eqref{eq:3.3} is in the space
\[
\mathcal C^2([0,\infty);L^2(B_T)^2)\cap \mathcal C^1([0,\infty);H^1(B_T\setminus\Gamma)^2)\cap \mathcal C([0,\infty);H^1_\Delta(B_T\setminus\Gamma)^2)
\]
and satisfies for all $t\ge 0$
\begin{eqnarray*}
|u(t)|_{\Delta,B_T\setminus\Gamma}+ |u^\star(t)|_{\Delta,B_T\setminus\Gamma} 
	& \le & C\big( H_3(\phi,t; H^{1/2}(\Gamma))+ H_2(\lambda,t;H^{-1/2}(\Gamma))\big),\\
\| \dot u(t)\|_{1,B_T\setminus\Gamma}+\|\dot u^\star(t)\|_{1,B_T\setminus\Gamma} 
	& \le & C\big( H_3(\phi,t; H^{1/2}(\Gamma))+ H_2(\lambda,t;H^{-1/2}(\Gamma))\big).
\end{eqnarray*}
\end{proposition}

\begin{proof}
The proof is very similar to that of Proposition \ref{prop:3.4}, using the sum of the liftings $L(0,\phi,0,0)$ and  $L(0,0,0,\lambda)$, plus the solution of two non-homogeneous differential equations of the second order associated to the operator $A$. Details are omitted.
\end{proof}

\section{Long term analysis in free space}\label{sec:5}

\paragraph{Causality arguments.} Assume that $f$ is a causal $\mathcal L(X,Y)$-valued distribution whose Laplace transform satisfies
\begin{equation}\label{eq:4.1}
\| \mathrm F(s)\| \le  e^{-M \, \mathrm{Re}\,s} C_{\mathrm F}(\mathrm{Re}\,s) | s|^\mu\qquad \forall s \in \mathbb C, \quad \mathrm{Re}\, s>0, 
\end{equation}
where $\mu \in \mathbb R$ and $C_{\mathrm F}:(0,\infty)\to (0,\infty)$ is a non-increasing function that can be bounded by a rational function close to zero. Then, if $g$ is an $X$-valued causal Laplace transformable distribution, $f*g$ is a $Y$-valued distribution whose support is contained in $[M,\infty)$. This is proved in \cite[Proposition 3.6.1]{SaySUB}.

\paragraph{Extension operators.} Consider a function $f:[0,\infty)\to X$ such that $\| f(t)\|_X$ is polynomially bounded in $t$. We can then define 
\[
(Ef)(t):=\left\{ \begin{array}{ll} 
f(t), & t\ge 0,\\
0, & t< 0,
\end{array}\right.
\]
which can be considered as a causal $X$-valued distribution. If $u:[0,\infty)\to L^2(B_T)$, we define
\[
\underline u(t):=\left\{ \begin{array}{ll}
u(t), & \mbox{in $B_T$},\\
0,     & \mbox{in $\mathbb R^d\setminus B_T$}.
\end{array}\right.
\]
We consider the quantity
\[
\delta:=\frac{\mathrm{dist}(\Gamma,\partial B_0)}{\max \{1,m\}},
\]
which measures the waiting time for a wave starting at $\Gamma$ and with speed equal to the maximum of the speeds of waves in \eqref{eq:3.1}, to reach the boundary of the first cut-off ball $B_0$.

\begin{proposition}\label{prop:5.1}
Let
\begin{subequations}
\begin{equation}
u\in \mathcal C^2([0,\infty);L^2(B_T))\cap \mathcal C^1([0,\infty);H^1(B_T\setminus\Gamma)) \cap \mathcal C([0,\infty);H^1_\Delta(B_T\setminus\Gamma))
\end{equation}
be polynomially bounded in time and satisfy
\begin{alignat}{6}
\ddot u(t)=\Delta u(t) &\qquad & t\ge 0,\\
\gamma_T u(t) = 0 && t\ge 0,\\
u(0)=\dot u(0)=0.
\end{alignat}
\end{subequations}
Then
\begin{equation}\label{eq:4.3}
\underline u(t)=\mathcal S * E\jump{\partial_\nu u} (t)-\mathcal D* E\jump{\gamma u}(t) \qquad \forall t\in [0,T+\delta]
\end{equation}
and
\begin{equation}
\underline u(t)\in \mathcal C([0,T+\delta];H^1_\Delta(\mathbb R^d\setminus\Gamma)).
\end{equation}
The same result holds if waves propagate at speed $m$ instead of unit speed.
\end{proposition}

\begin{proof}
Let
\begin{eqnarray*}
w_p &:=& \big( \mathcal S*E\jump{\partial_\nu u}-\mathcal D* E\jump{\gamma u}\big)|_{B_T},\\
w    &:=& Eu-w_p,\\
\xi   &:=& \gamma_T w=-\gamma_T w_p,\\
\mu &:=& \partial_T^\nu w= E\partial_T^\nu u-\partial_T^\nu w_p,
\end{eqnarray*}
where $\partial_T^\nu$ is used to denote the interior normal derivative on $\partial B_T$. If $\mathrm S:=\mathcal L\{ \mathcal S\}$ and $\mathrm D:=\mathcal L\{ \mathcal D\}$, then $\gamma_T \mathrm S$ and $\gamma_T \mathrm D$ can be shown to satisfy \eqref{eq:4.1} with $M=T+\delta$. This can be done using a direct argument on the kernels of these operators, and proves (see above) that
\begin{equation}\label{eq:4.5}
\mathrm{supp}\,\xi \subset [T+\delta,\infty).
\end{equation}
Similarly
\begin{equation}\label{eq:4.6}
\mathrm{supp}\,\partial_T^\nu w_p\subset [T+\delta,\infty).
\end{equation}
The distribution $w$ can then be shown to be a causal solution of
\begin{subequations}\label{eq:4.7}
\begin{alignat}{6}
\ddot w = \Delta w & \qquad & &\mbox{(in $L^2(B_T\setminus\Gamma)$)},\\
\jump{\gamma w} = 0 & &&\mbox{(in $H^{1/2}(\Gamma)$)},\\
\jump{\partial_\nu w} = 0 && &\mbox{(in $H^{-1/2}(\Gamma)$)},\\
\gamma_T w=\xi & & &\mbox{(in $H^{1/2}(\partial B_T)$)}.
\end{alignat}
\end{subequations}
However, \eqref{eq:4.7} can be considered as a uniquely solvable distributional evolution equation (this is done using Laplace transforms and elementary arguments) and therefore $\xi\mapsto \partial_T^\nu w$ can be shown to be a causal convolution operator. This and \eqref{eq:4.5} imply that $\mathrm{supp}\, \mu \subset [T+\delta,\infty)$, and then \eqref{eq:4.6} implies that $\mathrm{supp}\, \partial_T^\nu E u \subset [T+\delta,\infty)$. Therefore
\begin{equation}\label{eq:4.8}
\partial_T^\nu u(t)=0\qquad \forall t\in [0,T+\delta].
\end{equation}
Consider now the spatial extension $\underline u(t)$. Since the trace and normal derivative vanish on both sides of $\partial B_T$ up to $t=T+\delta$, it follows that
\[
\underline u(t)\in H^1_\Delta(\mathbb R^d\setminus\Gamma) \qquad t\in [0,T+\delta],
\]  
and in that time interval, the Laplacian operators in $\mathbb R^d\setminus\Gamma$ and $\mathbb R^d\setminus(\Gamma \cup\partial B_T)$ applied to $\underline u(t)$ give the same result. In the entire time interval
\[
\underline u\in \mathcal C^2([0,\infty);L^2(\mathbb R^d))\cap \mathcal C^1([0,\infty);H^1(\mathbb R^d))\cap \mathcal C([0,\infty);H^1_\Delta(\mathbb R^d\setminus(\Gamma\cup\partial B_T)))
\]
and
\[
\ddot{\underline u}(t)=\Delta \underline u(t), \qquad \underline u(0)=\underline{\dot u}(0)=0,
\]
which makes $E\underline u$ a causal solution of
\[
\frac{\mathrm d^2}{\mathrm dt^2}(E\underline u)=\Delta E\underline u \qquad \mbox{(in $H^1(\mathbb R^d\setminus(\Gamma\cup\partial B_T)$)}.
\]
Using the distributional Kirchhoff formula, we can write
\[
E\underline u=\mathcal S*E\jump{\partial_\nu u}-\mathcal D*E\jump{\gamma u}-\mathcal S_{\partial B_T} * E\partial_T^\nu u,
\]
where $\mathcal S_{\partial B_T}$ is a single layer potential emanating from $\partial B_T$. Finally, by \eqref{eq:4.8}, it follows that
\[
\mathrm{supp}\, (E\underline u-\mathcal S*E\jump{\partial_\nu u}+\mathcal D*E\jump{\gamma u})\subset [T+\delta,\infty),
\]
which proves \eqref{eq:4.3}.
\end{proof}

\paragraph{Recovery of densities.} So far the solution of \eqref{eq:3.1} and \eqref{eq:3.2}, or of \eqref{eq:3.1} and \eqref{eq:3.3} has been denoted without explicit reference to $T$. We now have to show that these solutions coincide on finite time intervals: they start being different once the outgoing wave hits the boundary of the cut-off domain. This will be the final step in reconciling the solutions of the cut-off problems with the distributional solutions of the problems in Section \ref{sec:2B}.

\begin{proposition}\label{prop:5.2}
Let $(u_T,u^\star_T)$ be the solution of \eqref{eq:3.1}-\eqref{eq:3.2}. Let then
\[
\psi(T):=\jump{\gamma u_T}(T), \quad \psi^\star(T):=\jump{\gamma u^\star_T}(T), \quad 
\mu(T):=\jump{\partial_\nu u_T}(T),\quad \mu^\star(T):=\jump{\partial_\nu u_T^\star}(T).
\]
Then for all $t\in [0,T+\delta]$
\[
\jump{\gamma u_T}(t)=\psi(t), \quad 
\jump{\gamma u_T^\star}(t)=\psi^\star(t),\quad
\jump{\partial_\nu u_T}(t)=\mu(t),\quad
\jump{\partial_\nu u_T^\star}(t)=\mu^\star(t).
\]
The same result holds for the solution of \eqref{eq:3.1} and \eqref{eq:3.3}.
\end{proposition}

\begin{proof}
Given $(u_T,u^\star_T)$ and $M>0$, we consider the extensions
\[
\underset{^\sim} u (t):=\left\{\begin{array}{ll} u_T(t), &\mbox{in $B_T$},\\ 0, & \mbox{in $B_{T+M}\setminus B_T$}
\end{array}\right.
\qquad
\underset{^\sim} u^\star(t):=\left\{\begin{array}{ll} u_T^\star(t), &\mbox{in $B_T$},\\ 0, & \mbox{in $B_{T+M}\setminus B_T$}.
\end{array}\right.
\]
With the argument given in the proof of Proposition \ref{prop:5.1} we can show that
\[
\underset{^\sim} u ,\underset{^\sim} u ^\star:[0,T+\delta]\to H^1_\Delta(B_{T+M}\setminus\Gamma)
\]
and then the pairs $(u_{T+M},u_{T+M}^\star)$ amd $(\underset{^\sim}u,\underset{^\sim}u^\star)$ solve equations \eqref{eq:3.1}-\eqref{eq:3.2} in the time interval $[0,T+\delta]$. A simple energy argument can then be used to show that they are equal, which implies that
\[
u_{T+M}|_{B_T}(t)\equiv u_T(t), \quad 
u_{T+M}^\star|_{B_T}(t)\equiv u_T^\star(t), \quad \forall t\in [0,T+\delta], \quad \forall T,M\ge 0.
\]
The result is a straightforward consequence of this fact.
\end{proof}

\begin{proof}[Proof of Theorem \ref{th:A}]
Assume first that $\beta_0|_{[0,\infty)}\in \mathcal C^3_0([0,\infty);H^{1/2}(\Gamma))$ and 
$\beta_1|_{[0,\infty)}\in \mathcal C^2_0([0,\infty);H^{-1/2}(\Gamma))$. Because of Proposition \ref{prop:3.4}, we have enough regularity to apply Proposition \ref{prop:5.1} to $u_T$ and $u_T^\star$ (changing the wave velocity in the latter case). We use then  the jumps defined in Proposition \ref{prop:5.2} and define the distributions
\[
u^h:=\mathcal S*E\mu-\mathcal D*E\psi, \qquad v^h:=\mathcal S_m*E\mu^\star-\mathcal D_m*E\psi^\star.
\]
Then the pair $(u^h,v^h)$ is the solution of \eqref{eq:2.11}, as follows from comparing the different transmission conditions in \eqref{eq:2.11} with those of \eqref{eq:3.1}-\eqref{eq:3.2}. Then Proposition \ref{prop:5.1} proves that
\[
\underline u_T(t)=u^h(t), \qquad \underline u^\star_T(t)=v^h(t), \qquad \forall t\in [0,T+\delta]
\]
and by Proposition \ref{prop:3.4}
\[
u^h,v^h\in \mathcal C^2([0,\infty);L^2(\mathbb R^d))
\cap \mathcal C^1([0,\infty);H^1(\mathbb R^d\setminus\Gamma))
\cap\mathcal C([0,\infty);H^1_\Delta(\mathbb R^d\setminus\Gamma))
\]
and
\begin{eqnarray*}
|u^h(t)|_{\Delta,\mathbb R^d\setminus\Gamma}+ |v^h(t)|_{\Delta,\mathbb R^d\setminus\Gamma} 
	& \le & C\big( H_3(\beta_0,t; H^{1/2}(\Gamma))+ H_2(\beta_1,t;H^{-1/2}(\Gamma))\big)\\
\| \dot u^h(t)\|_{1,\mathbb R^d\setminus\Gamma}+\|\dot v^h(t)\|_{1,\mathbb R^d\setminus\Gamma} 
	& \le & C\big( H_3(\beta_0,t; H^{1/2}(\Gamma))+ H_2(\beta_1,t;H^{-1/2}(\Gamma))\big).
\end{eqnarray*}
The remainder of the proof follows from simple shifting and density arguments. Details are identical to similar proofs  in \cite[Section 7.5]{Say13}.
\end{proof}

\begin{proof}[Proof of Theorem \ref{th:B}]
It is a slight variant of the previous proof.
\end{proof}

\section{Time discretization with Convolution Quadrature}\label{sec:6}

\paragraph{Full discretization.} The final step for discretization of our problem consists of applying one of the possible Convolution Quadrature techniques to the semidiscretized integral system \eqref{eq:2.9} (this includes using CQ in the convolution operators acting on the data functions) and to the potential postprocessing \eqref{eq:2.10}. We will apply a BDF2-based CQ discretization. Other multistep techniques (based on the implicit Euler scheme or on the trapezoidal rule) can be presented and analyzed with similar tools. Finally, multistage CQ is also available, using implicit Runge-Kutta methods as background ODE solvers in the numerical scheme. Multistep CQ originated in \cite{Lubich:1988}, applied to parabolic problems, and was extended in \cite{Lubich:1994} to hyperbolic problems, including a time domain boundary integral equation in acoustics. A modern introduction to computational uses of CQ for wave propagation problems is given in \cite{BaSc:2012}. Algorithmic details and several possible interpretations of CQ can be found in \cite{HaSa:2014}. In this section we will follow the plan developed in \cite[Section 10.3]{SaySUB}, based on \cite[Section 6]{BanLalSaySUB}. We note that while a time domain analysis of multistep CQ discretizations for wave propagation problems is known, at the current stage of research multistage CQ methods require a Laplace domain analysis of the associated Galerkin solver (see \cite{LalSay09}). In order to fix ideas, let us introduce the causal BDF2 approximation of the derivative
\[
\partial_k f:=\frac1k\left(\frac32 f -2 f(\cdot-k)+\frac12 f(\cdot-2k)\right),
\]
where $k$ will be the time-discretization parameter (time-step). Following \cite[Section 6]{BanLalSaySUB}, it is easy to prove that the BDF2-CQ discretization of \eqref{eq:2.9}-\eqref{eq:2.10} is equivalent to the transmission problem
\begin{subequations}\label{eq:6.1}
\begin{alignat}{6}
c^{-2} \partial_k^2 u^h_k=\kappa \Delta u^h_k &\qquad && \mbox{(in $L^2(\mathbb R^d\setminus\Gamma)$),}\\
\partial_k^2 v^h_k=\Delta v^h_k &\qquad && \mbox{(in $L^2(\mathbb R^d\setminus\Gamma)$),}\\
\jump{\gamma u^h_k}+\jump{\gamma v^h_k}=\beta_0 & & &\mbox{(in $H^{1/2}(\Gamma)$)},\\
\kappa \jump{\partial_\nu u^h_k}+\jump{\partial_\nu v^h_k}=\beta_1 & & &\mbox{(in $H^{-1/2}(\Gamma)$)},\\
\jump{\gamma u^h_k}\in Y_h & & & \mbox{(in $H^{1/2}(\Gamma)$)},\\
\jump{\partial_\nu u^h_k} \in X_h& & & \mbox{(in $H^{-1/2}(\Gamma)$)},\\
\gamma^+ u^h_k-\gamma^- v^h_k \in X_h^\circ & & & \mbox{(in $H^{1/2}(\Gamma)$)},\\
\kappa\partial_\nu^+ u^h_k-\partial_\nu^- v^h_k \in Y_h^\circ & & & \mbox{(in $H^{-1/2}(\Gamma)$)},
\end{alignat}
followed by the computation of
\begin{equation}
\lambda^h_k:=\jump{\partial_\nu u^h_k}, \qquad \phi^h_k:=\jump{\gamma u^h_k}.
\end{equation}
\end{subequations}
Even if the analysis is done through comparison of \eqref{eq:6.1} and \eqref{eq:2.11}-\eqref{eq:2.12}, let us emphasize two facts: (a) in reality what is solved is the integral system \eqref{eq:2.9}, which is followed by the potential postprocessing \eqref{eq:2.10}; (b) the solution is only computed at equally spaced time-steps of length $k$, and thus we are only computing  the values of the continuous causal functions $(\lambda^h_k,\phi^h_k, u^h_k,v^h_k)$ at  $t_n=k\,n$ for $n\ge 0$. 

\paragraph{Error equations.} Consider the errors for \eqref{eq:6.1} as a discretization of \eqref{eq:2.11}:
\[
e_u:=u^h-u^h_k, \quad e_v:=v^h-v^h_k, 
\]
The distributions $(e_u,e_v)$ are a causal solution to the error equations
\begin{subequations}\label{eq:6.2}
\begin{alignat}{6}
c^{-2} \partial_k^2 e_u=\kappa \Delta e_u+c^{-2}(\partial_k^2 u^h-\ddot u^h)&\qquad 
	&& \mbox{(in $L^2(\mathbb R^d\setminus\Gamma)$),}\\
\partial_k^2 e_v=\Delta e_v+(\partial_k^2 v^h-v^h) &\qquad 
	&& \mbox{(in $L^2(\mathbb R^d\setminus\Gamma)$),}\\
\jump{\gamma e_u}=-\jump{\gamma e_v}\in Y_h & & &\mbox{(in $H^{1/2}(\Gamma)$)},\\
\kappa \jump{\partial_\nu e_u}=-\jump{\partial_\nu e_v}\in X_h & & &\mbox{(in $H^{-1/2}(\Gamma)$)},\\
\gamma^+ e_u-\gamma^- e_v \in X_h^\circ & & & \mbox{(in $H^{1/2}(\Gamma)$)},\\
\kappa\partial_\nu^+ e_u-\partial_\nu^- e_v \in Y_h^\circ & & & \mbox{(in $H^{-1/2}(\Gamma)$)}.
\end{alignat}
\end{subequations}

\begin{proposition}\label{prop:6.1}
The errors $e_u$ and $e_u$ can be bounded for all $t\ge 0$ as follows
\begin{alignat*}{6}
\| \kappa^{1/2} \nabla e_u(t)\|_{\mathbb R^d\setminus\Gamma}
+\|\nabla e_v(t)\|_{\mathbb R^d}
+\|c^{-1}\partial_k e_u(t)\|_{\mathbb R^d}
+\|\partial_k e_v(t)\|_{\mathbb R^d}\qquad &\\
\le C k^2 t^2 \big(
	\max_{0\le \tau\le t} \|(\mathcal P_2 u^h)^{(4)}(\tau)\|_{\mathbb R^d}
	+\max_{0\le \tau\le t} \|(\mathcal P_2 v^h)^{(4)}(\tau)\|_{\mathbb R^d}
	\big),&
\end{alignat*}
where $\mathcal P_2 f:=f+2\dot f+\ddot f$.
\end{proposition}

\begin{proof}
Let $\mathrm E_u, \mathrm E_v, \mathrm U^h, \mathrm V^h$ be the respective Laplace transforms of $e_u, e_v, u^h,$ and $v^h$. Taking the Laplace transform of \eqref{eq:6.2}, and using an integration by parts argument (with the same format as in the proof of Proposition \ref{prop:3.1}(c)), it is easy to prove that for all $s\in \mathbb C$ with $\mathrm{Re}\,s>0$, it holds
\begin{alignat*}{6}
\mathrm{Re}\,s_k \Big(
	|s_k|^2 \| c^{-1}\mathrm E_u(s)\|_{\mathbb R^d}^2
	+\kappa \| \nabla \mathrm E_u(s)\|_{\mathbb R^d\setminus\Gamma}^2
	+|s_k|^2 \|\mathrm E_v(s)\|_{\mathbb R^d}^2
	+\|\nabla \mathrm E_v(s)\|_{\mathbb R^d\setminus\Gamma}^2\Big) \qquad & \\
= \mathrm{Re}\, 
		\Big(\overline{s_k} \Big( ( c^{-2}(s_k^2-s^2) \mathrm U^h(s),\overline{\mathrm E_u(s)})_{\mathbb R^d}+
		                    ( (s_k^2-s^2)\mathrm V^h(s),\overline{\mathrm E_v(s)})_{\mathbb R^d}\Big)\Big),&
\end{alignat*}
where $s_k=\frac32-2e^{-sk}+\frac12 e^{-2sk}$. Using that
\[
|s_k|\le C|s| \qquad |s_k^2-s^2|\le Ck^2|s|^4,\qquad
\mathrm{Re}\,s_k \ge C\min\{ 1,\mathrm{Re}\,s\} \qquad \forall k\le 1, \quad \mathrm{Re}\,s > 0,
\]
the bound
\begin{alignat*}{6}
	\| s_kc^{-1} \mathrm E_u(s)\|_{\mathbb R^d}
	+ \| \kappa^{1/2}\nabla \mathrm E_u(s)\|_{\mathbb R^d\setminus\Gamma}
	+\| s_k \mathrm E_v(s)\|_{\mathbb R^d}
	+\|\nabla \mathrm E_v(s)\|_{\mathbb R^d\setminus\Gamma}\qquad &\\
\le \frac{C k^2}{\min\{1,\mathrm{Re}\,s\}} \Big(
	\| s^4 \mathrm U^h(s)\|_{\mathbb R^d}+
	\| s^4 \mathrm V^h(s)\|_{\mathbb R^d}\Big)
\end{alignat*}
follows. The result in the statement is then a direct consequence of this inequality and \cite[Theorem 7.1]{DomSay13}.
\end{proof}

\begin{theorem}
The difference between the semidiscrete and the fully discrete solutions of \eqref{eq:2.7}-\eqref{eq:2.8} can be bounded for all $t\ge 0$
\begin{alignat*}{6}
t\| \nabla u^h(t)-\nabla u^h_k(t)\|_{\mathbb R^d\setminus\Gamma}+
t\| \nabla v^h(t)-\nabla v^h_k(t)\|_{\mathbb R^d\setminus\Gamma} 
+\|\phi^h(t)-\phi^h_k(t)\|_{1/2,\Gamma}\qquad & \\
\le C k^2 t^3 \Big( H_8(\beta_0,t\,;\, H^{1/2}(\Gamma))+ H_7(\beta_1,t\,;\, H^{-1/2}(\Gamma))\Big), & \\
\| \lambda^h(t)-\lambda^h_k(t) \|_{-1/2,\Gamma} \le 
C k^2 t^2 \Big( H_9(\beta_0,t\,;\, H^{1/2}(\Gamma))+ H_8(\beta_1,t\,;\, H^{-1/2}(\Gamma))\Big). &
\end{alignat*}
\end{theorem}

\begin{proof}
The bounds for the gradients follows from Proposition \ref{prop:6.1} and Theorem \ref{th:A}. In order to bound the traces  we use the fact that for a causal $X$-valued function $f$
\[
\| f(t)\|_X \le k \sum_{j=0}^\infty \| \partial_k f(t-t_j)\|_X\le t\max_{0\le \tau\le t} \| f(\tau)\|_X
\]
and then use this to obtain $L^2(\mathbb R^d)$ estimates of $e_v(t)$ from those of $\partial_k e_v(t)$ given in Proposition \ref{prop:6.1}. Finally, in order to bound normal derivatives of $e_v$, we use the error equation $\Delta e_v=\partial_k^2 e_v+(\ddot v^h-\partial_k^2 v^h)$ to bound $\Delta e_v(t)$. The details are very similar to those of the proof of \cite[Theorem 6.7]{BanLalSaySUB} (see also \cite[Section 10.3]{SaySUB}) and are therefore omitted.
\end{proof}

\section{Numerical experiments}\label{sec:7}

\paragraph{A smooth obstacle.} In this experiment we choose $\Gamma\subset \mathbb R^2$ to be the smooth closed curve parametrized by the function
\[
z\mapsto \left( (1+(\cos z)^2)\cos z, \, (1+(\sin z)^2)\sin z\right) 
\left(\begin{array}{cc} 1/\sqrt2 & 1/\sqrt2 \\ - 1/\sqrt2 & 1/\sqrt2
\end{array}\right),
\]
which is shaped like a smoothened square parallel to the coordinate axes. We choose $\kappa=0.8$ and $c^2\kappa=1.2^2$. The data in \eqref{eq:2.1} are taken so that the exterior solution is zero and the interior solution is a plane wave $u(\mathbf x,t)=\sin (c\sqrt\kappa (t-t_0)-\mathbf x\cdot\mathbf d) \, h(c\sqrt\kappa (t-t_0)-\mathbf x\cdot\mathbf d)$, where $t_0=2.2$, $\mathbf d=(1/\sqrt2,-1/\sqrt2)$ and $h$ is a smoothened version of the Heaviside function, namely a polynomial of degree ten that connects the points $(0,0)$ and $(1,1)$. We will integrate in the time interval $[0,4]$ using the BDF2-based CQ scheme.

For spatial discretization we use a division of $\Gamma$ into $N$ elements, by choosing a uniform grid in parametric space. The space $X_h$ is composed of piecewise constant functions in this mesh. For $Y_h$ we choose continuous piecewise linear functions in the parameter $z$, mapped to $\Gamma$ and defined on a uniform grid with $N$ elements which is staggered with respect to the grid that is used to define $X_h$. With adequately chosen reduced integration, it is possible to rewrite all the elements of the matrices in the language of the fully discrete Calder\'on calculus of \cite{DoLuSa:2014}. The interior solution will be computed in the center and corners of the square $[-0.5,0.5]^2 \subset \Omega_-$: these five points will be denoted $\mathbf x_\ell^{\mathrm{obs}}$ for $\ell=1,\ldots,5.$ We then measure relative errors associated to the absolute errors
\begin{eqnarray*}
e^\lambda &:=& \| \lambda(T)-\Pi_0\kappa \partial_\nu u(T)\|_{L^\infty(\Gamma)},
\\
e^\varphi &:=&\| \varphi(T)-\Pi_1\gamma u(T)\|_{L^\infty(\Gamma)},
\\
e^u &:=& \max_\ell | u^h(\mathbf x_\ell^{\mathrm{obs}},T)-u(\mathbf x_\ell^{\mathrm{obs}},T)|. 
\end{eqnarray*}
(The corresponding relative errors will be denoted with a capital $E$ in the tables.)
Here $\Pi_0$ is the midpoint interpolation operator on $X_h$ and $\Pi_1$ is the natural Lagrange interpolation on $Y_h$. 
The results are reported in Table \ref{table:1} and Figure \ref{figure:1}.

\begin{table}[h]
\begin{tabular}{r|cc|cc|cc}
\hline
$N=M$  & $E^\varphi$     & e.c.r. & $E^\lambda$   & e.c.r. & $E^u$ & e.c.r. \\ \hline
50   & 3.0159E(-01) &        & 1.0925E(+00) &        & 1.6530E(-01)   &        \\
100  & 2.2310E(-01) & 0.4349 & 7.8330E(-01) & 0.4800 & 4.1091E(-02)   & 2.0082 \\
150  & 1.7356E(-01) & 0.6192 & 6.8195E(-01) & 0.3418 & 1.8436E(-02)   & 1.9767 \\
200  & 1.4065E(-01) & 0.7308 & 6.1521E(-01) & 0.3580 & 1.0385E(-02)   & 1.9952 \\
300  & 9.7445E(-02) & 0.9052 & 5.2374E(-01) & 0.3970 & 4.5692E(-03)   & 2.0248 \\
500  & 5.0972E(-02) & 1.2686 & 3.0757E(-01) & 1.0420 & 1.7325E(-03)   & 1.8984 \\
700  & 2.9493E(-02) & 1.6260 & 1.7386E(-01) & 1.6954 & 8.5386E(-04)   & 2.1029 \\
900  & 1.8488E(-02) & 1.8584 & 1.0393E(-01) & 2.0475 & 4.8647E(-04)   & 2.2386 \\
1200 & 1.0462E(-02) & 1.9791 & 5.9202E(-02) & 1.9560 & 2.6762E(-04)   & 2.0774 \\
1600 & 5.8992E(-03) & 1.9917 & 3.2909E(-02) & 2.0412 & 1.4956E(-04)   & 2.0226 \\ \hline
\end{tabular}
\caption{Relative errors at final time ($T=4$) for trace, normal derivative and interior solution in the case of a smooth obstacle.}\label{table:1}
\end{table}

\begin{figure}[h]
\begin{center}
\includegraphics[width=9cm]{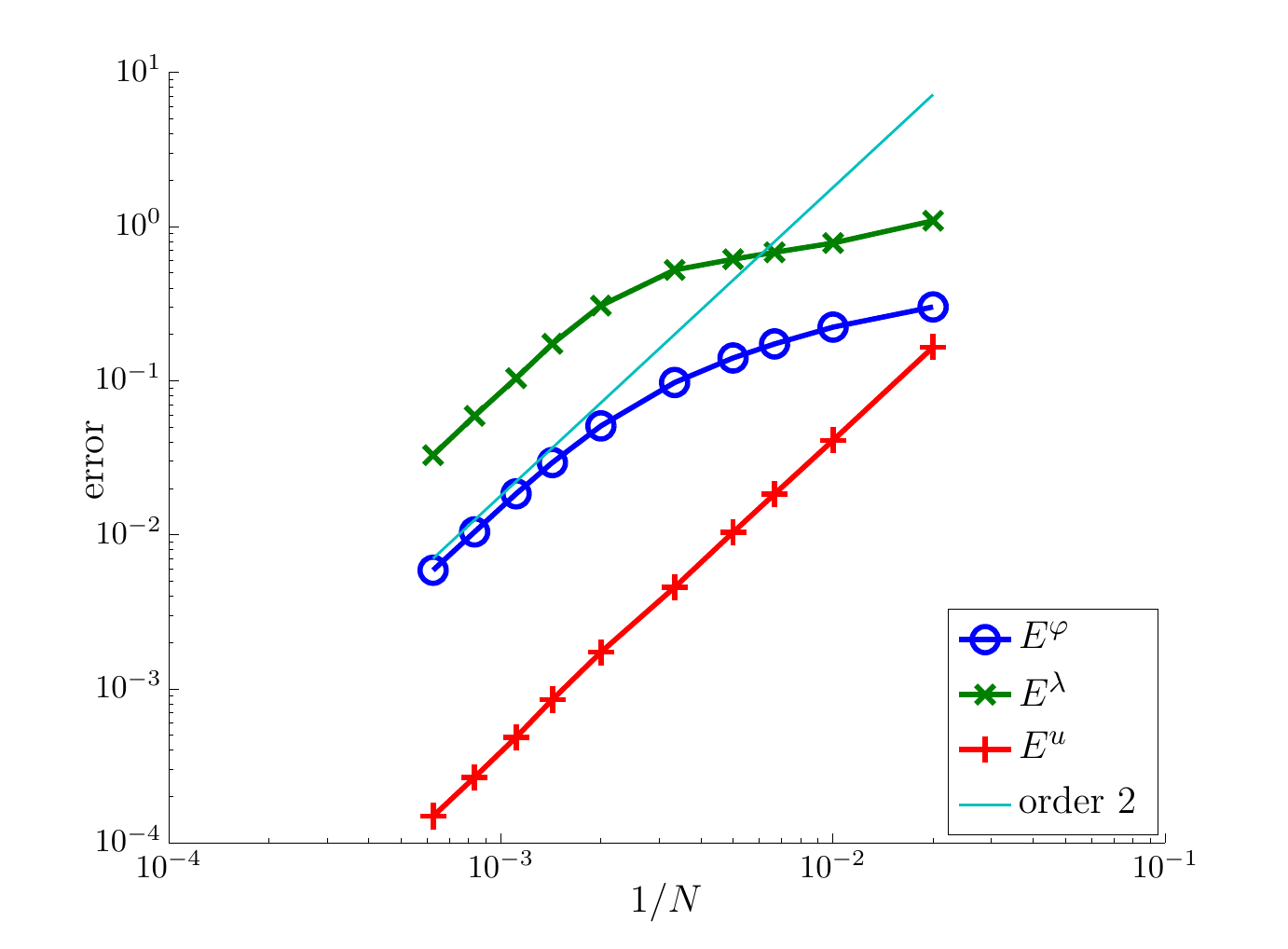}
\end{center}
\caption{Error graphs corresponding to Table \ref{table:1}}\label{figure:1}
\end{figure}

\paragraph{A polygonal obstacle.} We now consider an obstacle whose boundary is the quadrilateral with vertices $(0,0)$, $(1,0)$, $(0.8,0.8)$, and $(0.2,1)$. The same physical data ($\kappa$ and $c$) and the exact solution are the same as in the previous experiment, and we also integrate up to time $T=4$ with the BDF2-CQ scheme. Each of the edges of $\Gamma$ is subdivided into an equal number of elements (the partition is thus piecewise uniform) to a total of $N$ elements. The spaces $X_h$ and $Y_h$ are respectively composed of piecewise constant and continuous piecewise linear functions on this grid. The interior solution will be observed in the points
\[
\mathbf x_1^{\mathrm{obs}}:=(0.3,0.4),
\quad
\mathbf x_2^{\mathrm{obs}}:=(0.5,0.7),
\quad
\mathbf x_3^{\mathrm{obs}}:=(0.65,0.4),
\quad
\mathbf x_4^{\mathrm{obs}}:=(0.5,0.2).
\]
We measure relative errors corresponding to
\begin{eqnarray*}
e^\lambda &:=& \| \lambda(T)-\kappa \partial_\nu u(T)\|_{L^2(\Gamma)},
\\
e^\varphi &:=&\| \varphi(T)-\gamma u(T)\|_{L^2(\Gamma)},
\\
e^u &:=& \max_\ell | u^h(\mathbf x_\ell^{\mathrm{obs}},T)-u(\mathbf x_\ell^{\mathrm{obs}},T)|. 
\end{eqnarray*}

\begin{table}[h]
\begin{tabular}{rr|cc|cc|cc}
\hline
N   & M    & $E^\varphi$     & e.c.r  & $E^\lambda$  & e.c.r  & $E^u$ & e.c.r  \\ \hline
4   & 150  & 4.9350E(-02) &        & 6.0959E(-02) &        & 1.5495E(-02)   &        \\
8  & 300  & 4.8955E(-03) & 3.3335 & 4.1232E(-02) & 0.5641 & 2.8475E(-03)   & 2.4441 \\
16  & 600  & 1.2233E(-03) & 2.0007 & 1.3734E(-02) & 1.5860 & 8.5962E(-04)   & 1.7279 \\
32  & 1200 & 8.6127E(-05) & 3.8281 & 1.7186E(-03) & 2.9984 & 9.4632E(-05)   & 3.1833 \\
64 & 2400 & 1.9057E(-05) & 2.1762 & 4.1594E(-04) & 2.0468 & 2.6609E(-05)   & 1.8304 \\
128 & 4800 & 4.6711E(-06) & 2.0284 & 1.0252E(-04) & 2.0205 & 6.7798E(-06)   & 1.9726 \\ \hline
\end{tabular}
\caption{Relative errors at final time for trace, normal derivative and interior solution in the case of a non-smooth scatterer.}\label{table:2}
\end{table}

\begin{figure}[h]
\begin{center}
\includegraphics[width=9cm]{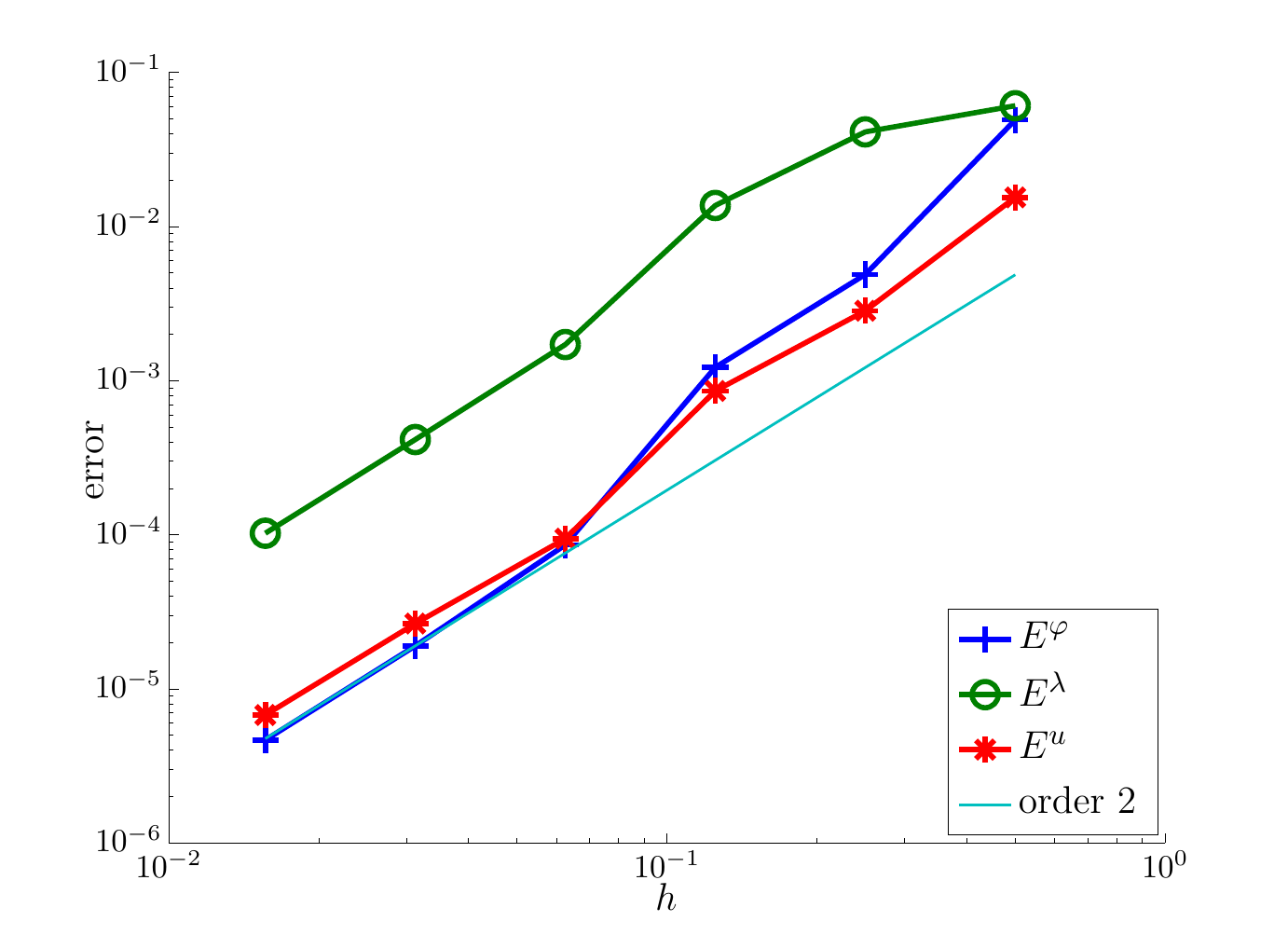}
\end{center}
\caption{Error graphs corresponding to Table \ref{table:2}.}\label{figure:2}
\end{figure}

\paragraph{A multiple scatterer illustration.} We finally show several snapshots of the scattering of a plane wave by four circular obstacles with different material properties. In all four obstacles $\kappa=1$. The wave velocity is set to be $c=2$ in the obstacles placed in NE and SW positions (see Figure \ref{figure:3}) and $c=0.5$ in the other two obstacles.

\begin{figure}[h]
\begin{center}
\begin{tabular}{cc}
\includegraphics[width=7cm]{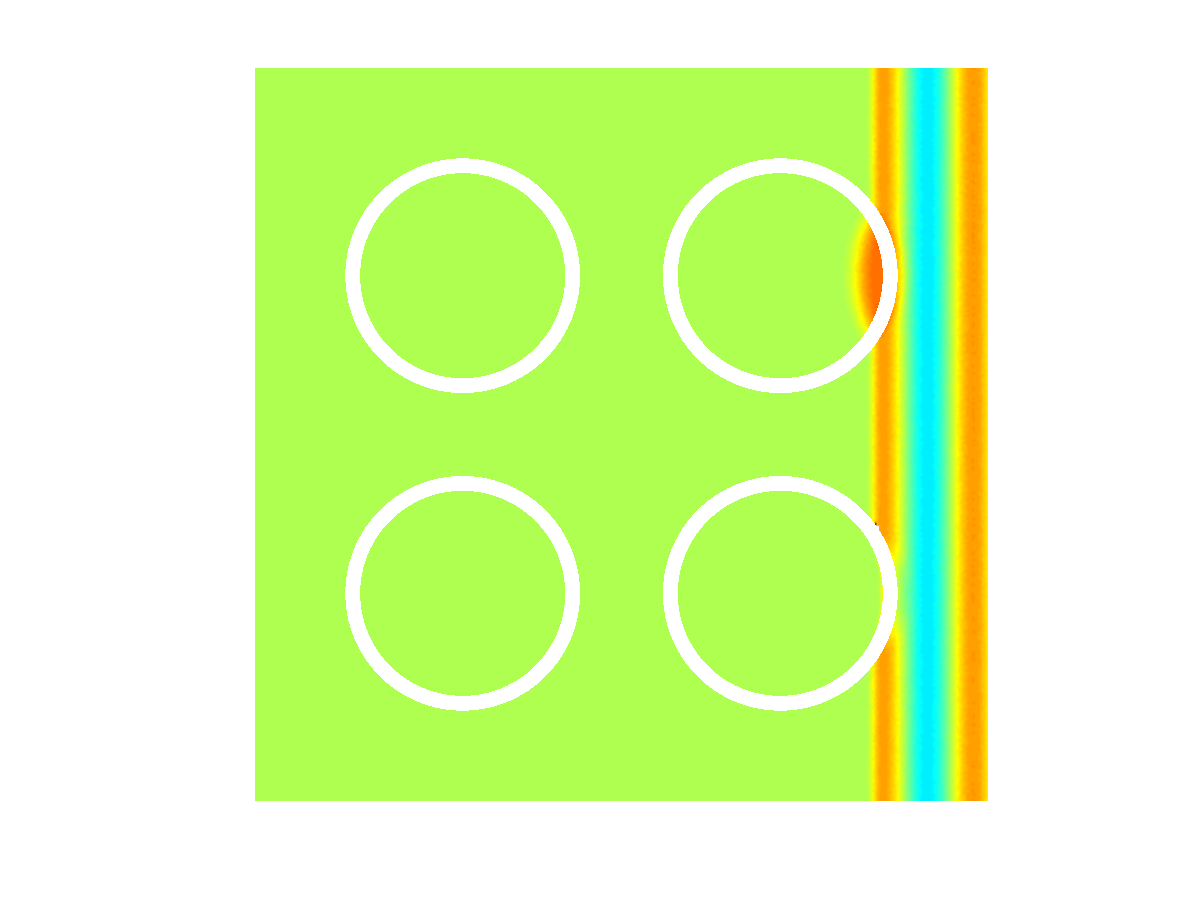} & \includegraphics[width=7cm]{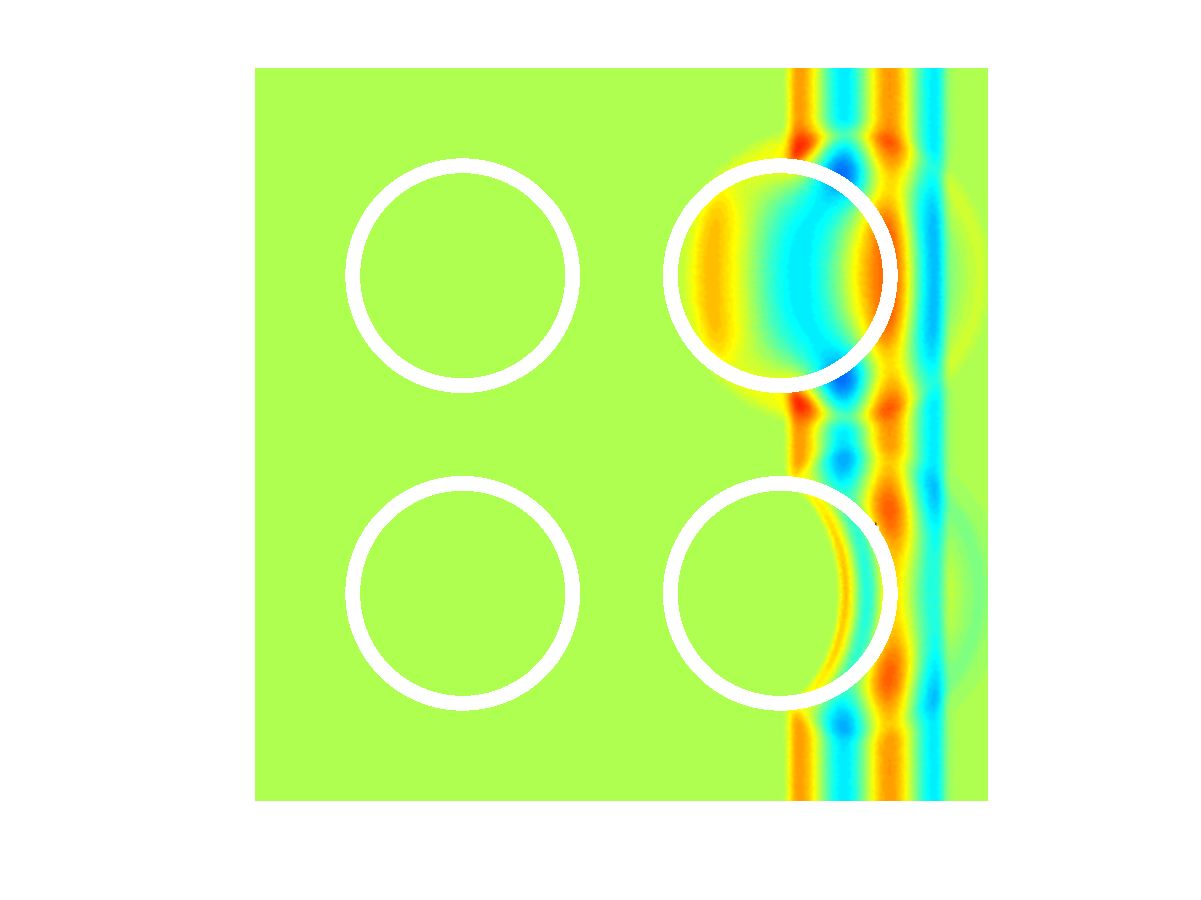} \\
\includegraphics[width=7cm]{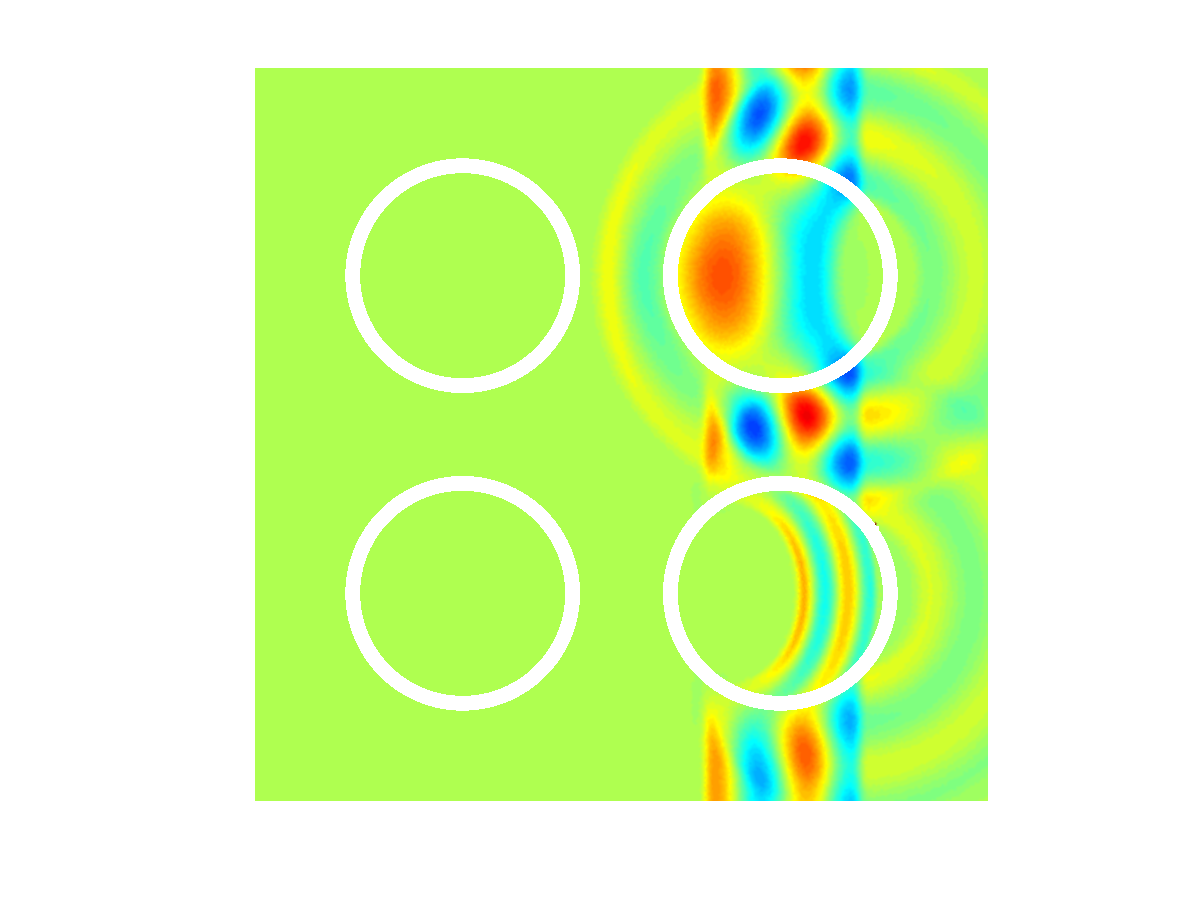} & \includegraphics[width=7cm]{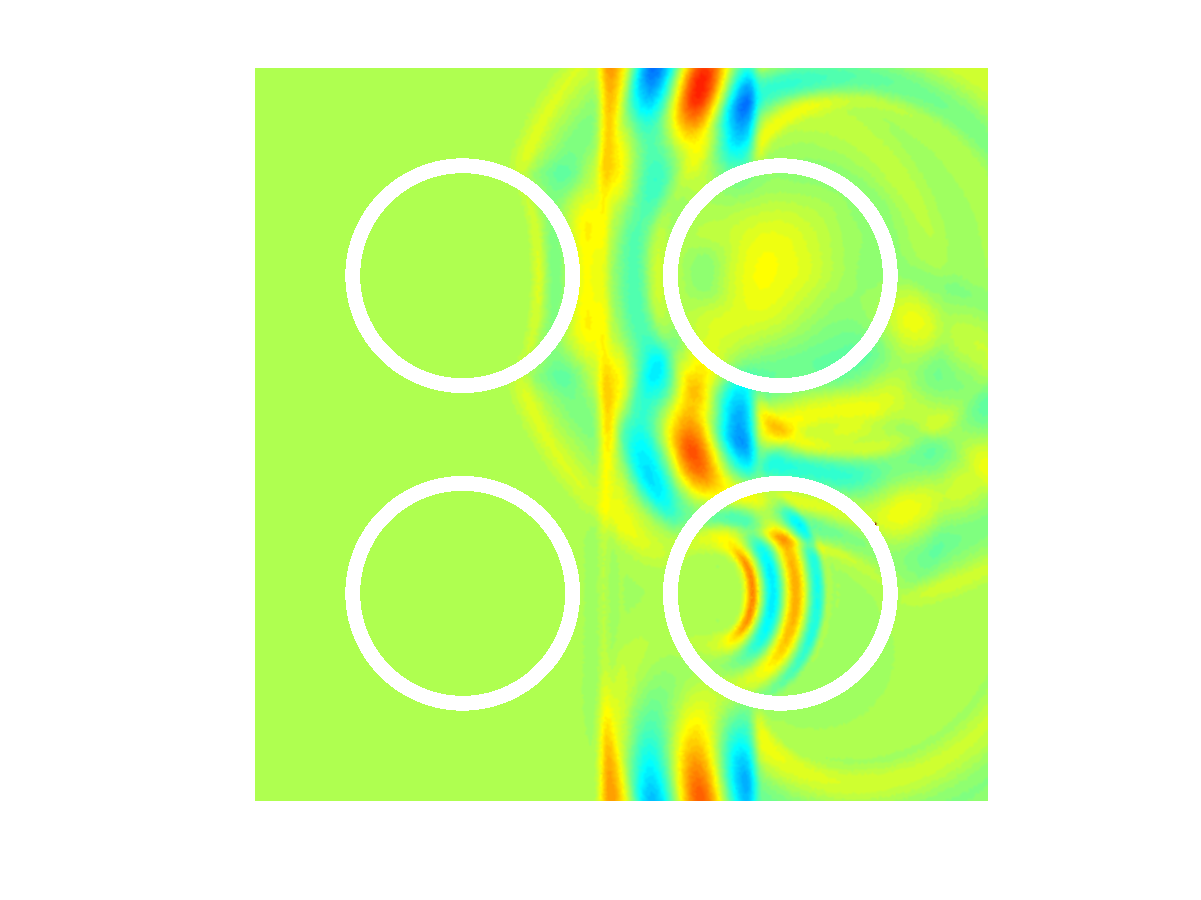} \\
\includegraphics[width=7cm]{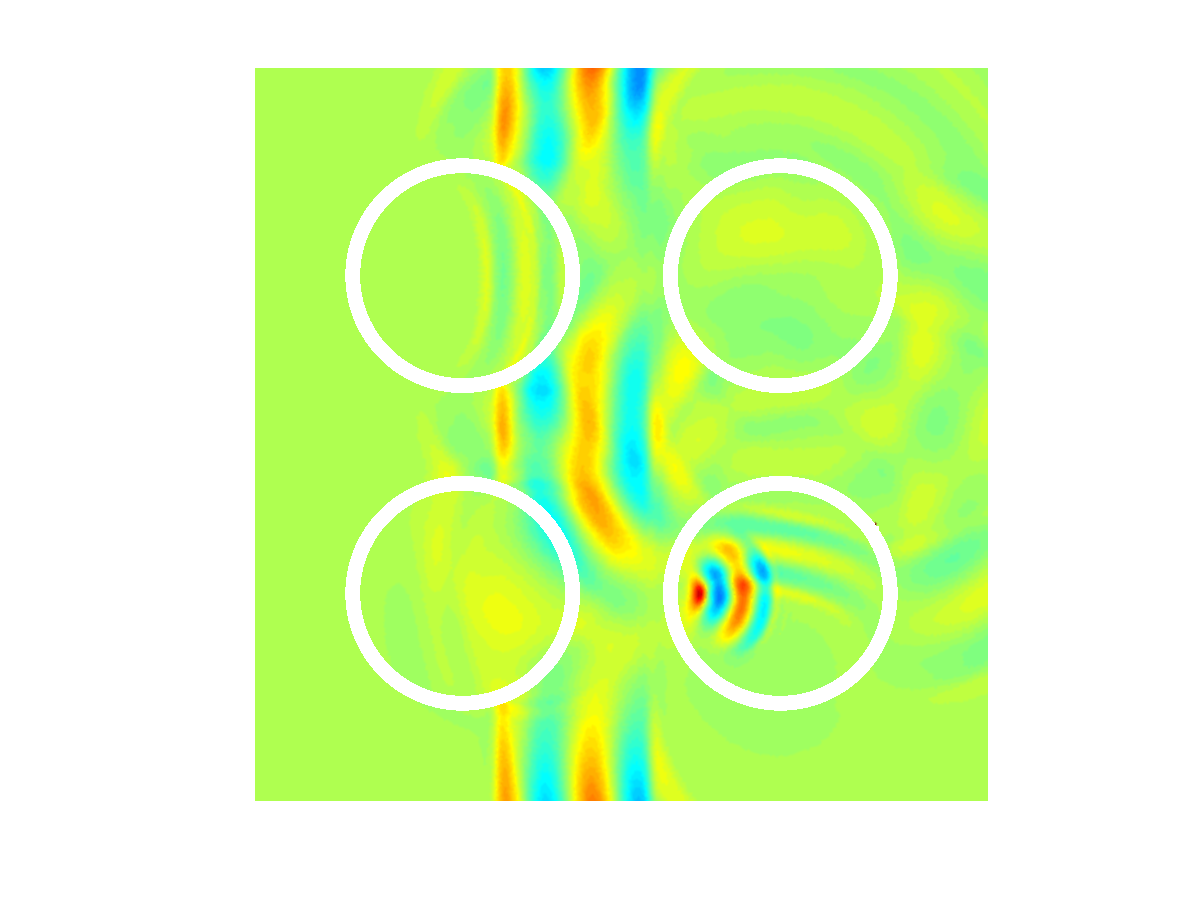} & \includegraphics[width=7cm]{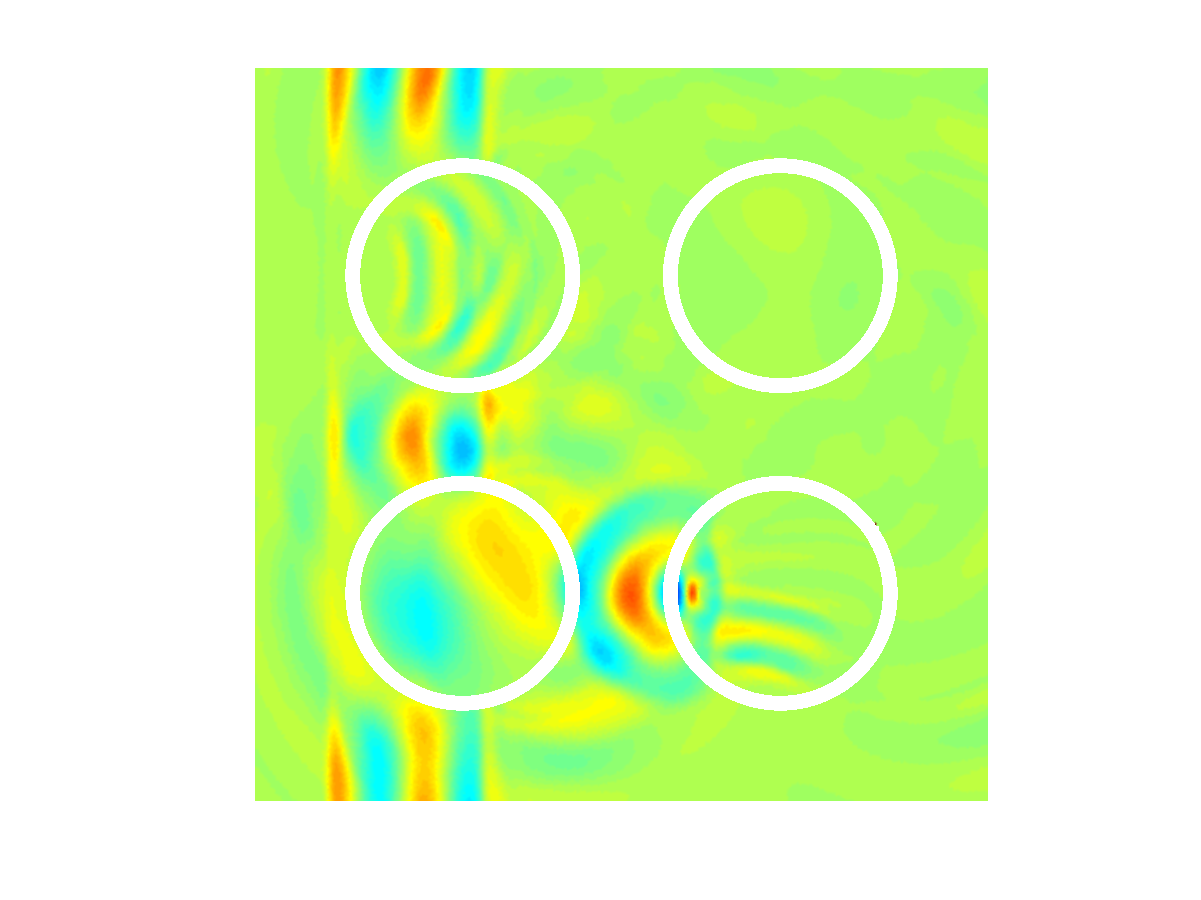} 
\end{tabular}
\end{center}
\caption{Six images of the scattering of a short-pulse plane wave by four penetrable obstacles. The obstacle in the upper right corner propagates waves at twice the speed of the surrounding medium, giving a head start to the part of the wave that traverses the obstacle with respect to the incident wave. }\label{figure:3}
\end{figure}

\bibliographystyle{abbrv}
\bibliography{RefsTDIE}

\begin{thebibliography}{10}

\bibitem{AbJoRoTe:2011}
T.~Abboud, P.~Joly, J.~Rodr{\'{\i}}guez, and I.~Terrasse.
\newblock Coupling discontinuous {G}alerkin methods and retarded potentials for
  transient wave propagation on unbounded domains.
\newblock {\em J. Comput. Phys.}, 230(15):5877--5907, 2011.

\bibitem{BaHa:1986a}
A.~Bamberger and T.~H. Duong.
\newblock Formulation variationnelle espace-temps pour le calcul par potentiel
  retard\'e de la diffraction d'une onde acoustique. {I}.
\newblock {\em Math. Methods Appl. Sci.}, 8(3):405--435, 1986.

\bibitem{BaHa:1986b}
A.~Bamberger and T.~H. Duong.
\newblock Formulation variationnelle pour le calcul de la diffraction d'une
  onde acoustique par une surface rigide.
\newblock {\em Math. Methods Appl. Sci.}, 8(4):598--608, 1986.

\bibitem{Banjai:2010}
L.~Banjai.
\newblock Multistep and multistage convolution quadrature for the wave
  equation: algorithms and experiments.
\newblock {\em SIAM J. Sci. Comput.}, 32(5):2964--2994, 2010.

\bibitem{BanLalSaySUB}
L.~Banjai, A.~Laliena, and F.-J. Sayas.
\newblock Fully discrete {K}irchhoff formulas using {CQ} and {BEM}.
\newblock To appear in {\em IMA J. Numer. Anal.}

\bibitem{BaLuSa:2014}
L.~Banjai, C.~Lubich, and F.-J. Sayas.
\newblock Stable numerical coupling of exterior and interior problems for the
  wave equation, 2014.
\newblock To appear in {\em Numer. Math.}

\bibitem{BaMeSc:2012}
L.~Banjai, M.~Messner, and M.~Schanz.
\newblock Runge-{K}utta convolution quadrature for the boundary element method.
\newblock {\em Comput. Methods Appl. Mech. Engrg.}, 245/246:90--101, 2012.

\bibitem{BaSc:2012}
L.~Banjai and M.~Schanz.
\newblock Wave propagation problems treated with convolution quadrature and
  {BEM}.
\newblock In {\em Fast boundary element methods in engineering and industrial
  applications}, volume~63 of {\em Lect. Notes Appl. Comput. Mech.}, pages
  145--184. Springer, Heidelberg, 2012.

\bibitem{BoDoLeTu:2013}
Y.~Boubendir, V.~Dominguez, D.~Levadoux, and C.~Turc.
\newblock Regularized combined field integral equations for acoustic
  transmission problems, 2013.
\newblock arXiv:1312.6598.

\bibitem{ChMo:2014}
J.~F.-C. Chan and P.~Monk.
\newblock Time dependent electromagnetic scattering by a penetrable obstacle,
  2014.
\newblock To appear in {\em BIT.}

\bibitem{CheMonSB}
Q.~Chen and P.~Monk.
\newblock Discretization of the time domain cfie for acoustic scattering
  problems using convolution quadrature.
\newblock Submitted.

\bibitem{ClHi:2013}
X.~Claeys and R.~Hiptmair.
\newblock {Multi-Trace Boundary Integral Formulation for Acoustic Scattering by
  Composite Structures}.
\newblock {\em Communications on Pure and Applied Mathematics},
  66(8):1163--1201, 2013.

\bibitem{CoSt:1985}
M.~Costabel and E.~Stephan.
\newblock {A direct boundary integral equation method for transmission
  problems}.
\newblock {\em Journal of Mathematical Analysis and Applications},
  106(2):367--413, 1985.

\bibitem{DoLuSa:2014}
V.~Dom\'{\i}nguez, S.~L. Lu, and F.-J. Sayas.
\newblock {A Nystr\"{o}m flavored Calder\'{o}n Calculus of order three for two
  dimensional waves, time-harmonic and transient}.
\newblock {\em Computers \& Mathematics with Applications}, 67(1):217--236,
  Jan. 2014.

\bibitem{DomSay13}
V.~Dom{\'{\i}}nguez and F.-J. Sayas.
\newblock Some properties of layer potentials and boundary integral operators
  for the wave equation.
\newblock {\em J. Integral Equations Appl.}, 25(2):253--294, 2013.

\bibitem{FaMo:2014}
S.~Falletta and G.~Monegato.
\newblock An exact non reflecting boundary condition for 2{D} time-dependent
  wave equation problems.
\newblock {\em Wave Motion}, 51(1):168--192, 2014.

\bibitem{GrHeLa:2013}
S.~P. Groth, D.~P. Hewett, and S.~Langdon.
\newblock {Hybrid numerical-asymptotic approximation for high-frequency
  scattering by penetrable convex polygons}.
\newblock To appear in {\em IMA Journal of Applied Mathematics}.

\bibitem{HaSa:2014}
M.~Hassell and F.-J. Sayas.
\newblock Convolution quadrature for wave simulations, 2014.
\newblock arXiv:1407.0345.

\bibitem{HiJe:2011}
R.~Hiptmair and C.~Jerez-Hanckes.
\newblock {Multiple traces boundary integral formulation for Helmholtz
  transmission problems}.
\newblock {\em Advances in Computational Mathematics}, 37(1):39--91, 2011.

\bibitem{KlMa:1988}
R.~E. Kleinman and P.~A. Martin.
\newblock {On Single Integral Equations for the Transmission Problem of
  Acoustics}.
\newblock {\em SIAM Journal on Applied Mathematics}, 48(2):307--325, 1988.

\bibitem{KrRo:1978}
R.~Kress and G.~F. Roach.
\newblock {Transmission problems for the Helmholtz equation}.
\newblock {\em Journal of Mathematical Physics}, 19(6):1433, 1978.

\bibitem{LalSay09}
A.~R. Laliena and F.-J. Sayas.
\newblock Theoretical aspects of the application of convolution quadrature to
  scattering of acoustic waves.
\newblock {\em Numer. Math.}, 112(4):637--678, 2009.

\bibitem{Lubich:1988}
C.~Lubich.
\newblock Convolution quadrature and discretized operational calculus. {I}.
\newblock {\em Numer. Math.}, 52(2):129--145, 1988.

\bibitem{Lubich:1994}
C.~Lubich.
\newblock On the multistep time discretization of linear initial-boundary value
  problems and their boundary integral equations.
\newblock {\em Numer. Math.}, 67(3):365--389, 1994.

\bibitem{Say13}
F.-J. Sayas.
\newblock Energy estimates for {G}alerkin semidiscretizations of time domain
  boundary integral equations.
\newblock {\em Numer. Math.}, 124(1):121--149, 2013.

\bibitem{SaySUB}
F.-J. Sayas.
\newblock Retarded potentials and time domain integral equations: a roadmap,
  2014.
\newblock Submitted.

\bibitem{ToWe:1993}
R.~H. Torres and G.~V. Welland.
\newblock {The Helmholtz-equation and transmission problems with Lipschitz
  interfaces}.
\newblock {\em Indiana University Mathematics Journal}, 42(4):1457--1485, 1993.

\bibitem{Petersdorff:1989}
T.~{Von Petersdorff}.
\newblock {Boundary integral equations for mixed Dirichlet, Neumann and
  transmission problems}.
\newblock {\em Mathematical Methods in the Applied Sciences}, 11(2):185--213,
  1989.

\bibitem{Zinn:1989}
A.~Zinn.
\newblock {A numerical method for transmission problems for the Helmholtz
  equation}.
\newblock {\em Computing}, 41(3):267--274, 1989.

\end{thebibliography}

\end{document}